\documentclass[12pt]{amsart}  

\usepackage[latin1]{inputenc}
\usepackage{amsmath, bm, mathdots} 
\usepackage{amsfonts}
\usepackage{amssymb}
\usepackage{stmaryrd}
\usepackage{latexsym} 
\usepackage{graphicx}
\usepackage{subfigure}
\usepackage{xcolor}
\usepackage{hyperref}
\usepackage{verbatim}
\usepackage[all]{xy}
\usepackage{graphics}
\usepackage{pdfsync}
\usepackage{tikz}
\usepackage{url}
\usepackage{enumerate}

\newcommand{\SXL}[1]{\textcolor{black}{#1}}
\newcommand{\FA}[1]{\textcolor{black}{#1}}
\newcommand{\svw}[1]{\textcolor{black}{#1}}
\newcommand{\sagan}[1]{\textcolor{black}{#1}}
\newcommand{\svwtwo}[1]{\textcolor{black}{#1}}
\newcommand{\svwrev}[1]{\textcolor{black}{#1}}
\newcommand{\svwrevtwo}[1]{\textcolor{black}{#1}}

\theoremstyle{definition}
\newtheorem{theorem}{Theorem}[section]
\newtheorem*{theorem*}{Theorem}
\newtheorem{proposition}[theorem]{Proposition}

\newtheorem{lemma}[theorem]{Lemma}
\newtheorem{corollary}[theorem]{Corollary}

\newtheorem{definition}[theorem]{Definition}
\newtheorem{example}[theorem]{Example}

\theoremstyle{remark}

\numberwithin{equation}{section}
\newcommand{\bQ}{\mathbb{Q}}

\newcommand{\suchthat}{\,:\,}
\newcommand{\spam}{\operatorname{span}}

\newcommand{\Sym}{\ensuremath{\operatorname{Sym}}}
\newcommand{\Sn}{\mathfrak{S}}

\newcommand{\NSym}{\ensuremath{\operatorname{NSym}}}
\newcommand{\nch}{\mathbf{h}}         	
\newcommand{\ncr}{\mathbf{r}}           	
  
\newcommand{\jota}{\mathcal{J}}    

\newcommand{\NCSym}{\ensuremath{\operatorname{NCSym}}}
\newcommand{\slashp}{\mid}

\newcommand{\JT}{\mathrm{JT}}
\newcommand{\sh}{\mathrm{sh}}
\newcommand{\bdet}{\bm{\mathrm{det}}}
\newcommand{\id}{\mathrm{id}}
\newcommand{\st}{s^t}
\newcommand{\sort}{{\rm sort}}
\newcommand{\hig}{{\rm ht}}
\newcommand{\exch}{{\xi}}

\newlength\cellsize 
\savebox2{%
\begin{picture}(15,15)
\put(0,0){\line(1,0){15}}
\put(0,0){\line(0,1){15}}
\put(15,0){\line(0,1){15}}
\put(0,15){\line(1,0){15}}
\end{picture}}
\newcommand\cellify[1]{\def\thearg{#1}\def\nothing{}%
\ifx\thearg\nothing
\vrule width0pt height\cellsize depth0pt\else
\hbox to 0pt{\usebox2\hss}\fi%
\vbox to 15\unitlength{
\vss
\hbox to 15\unitlength{\hss$#1$\hss}
\vss}}
\newcommand\tableau[1]{\vtop{\let\\=\cr
\setlength\baselineskip{-16000pt}
\setlength\lineskiplimit{16000pt}
\setlength\lineskip{0pt}
\halign{&\cellify{##}\cr#1\crcr}}}
\savebox3{%
\begin{picture}(15,15)
\put(0,0){\line(1,0){15}}
\put(0,0){\line(0,1){15}}
\put(15,0){\line(0,1){15}}
\put(0,15){\line(1,0){15}}
\end{picture}}
\newcommand\expath[1]{%
\hbox to 0pt{\usebox3\hss}%
\vbox to 15\unitlength{
\vss
\hbox to 15\unitlength{\hss$#1$\hss}
\vss}}
\newcommand\bas[1]{\omit \vbox to \cellsize{ \vss \hbox to \cellsize{\hss$#1$\hss} \vss}}

\begin{document}

\title[Schur functions in $\NCSym$]{Schur functions in noncommuting variables}

\author{Farid Aliniaeifard}
\address{
 Department of Mathematics,
 University of British Columbia,
 Vancouver BC V6T 1Z2, Canada}
\email{farid@math.ubc.ca}

\author{Shu Xiao Li}
\address{
School of Mathematical Sciences,
Dalian University of Technology,
Dalian Liaoning 116024, P.R. China}
\email{lishuxiao@dlut.edu.cn}

\author{Stephanie van Willigenburg}
\address{
 Department of Mathematics,
 University of British Columbia,
 Vancouver BC V6T 1Z2, Canada}
\email{steph@math.ubc.ca}

\thanks{All authors were supported  in part by the National Sciences and Engineering Research Council of Canada.}

\begin{abstract} In 2004 Rosas and Sagan asked whether there was a way to define a basis in the algebra of symmetric functions in noncommuting variables, $\NCSym$, having properties analogous to the classical  Schur functions. \svwtwo{This was because they had constructed a partial such set that was not a basis. We answer their} question by defining Schur functions in noncommuting variables using a noncommutative analogue of the Jacobi-Trudi determinant. Our Schur functions in $\NCSym$ map to classical Schur functions under commutation, and a subset of them indexed by set partitions forms a basis for $\NCSym$.  \svwrevtwo{Amongst} other properties, Schur functions in $\NCSym$ also satisfy a noncommutative analogue of the product rule for classical Schur functions in terms of skew Schur functions.

We also show how Schur functions in $\NCSym$ are related to Specht modules, and naturally refine the Rosas-Sagan Schur functions. Moreover, by generalizing Rosas-Sagan Schur functions to skew Schur functions in the natural way, we prove noncommutative analogues of the Littlewood-Richardson rule and coproduct rule for them. \FA{Finally, we relate our functions to noncommutative symmetric functions by proving a subset of our functions are natural extensions of noncommutative ribbon Schur functions, and immaculate functions indexed by integer partitions.}
\end{abstract}

\maketitle
\tableofcontents

\section{Introduction}\label{sec:intro}  
Symmetric functions in noncommuting variables, $\NCSym$, were introduced in 1936 by Wolf \cite{Wolf} who produced a noncommutative analogue of the fundamental theorem of symmetric functions. Her work was generalized in 1969 by Bergman and Cohn \cite{BC} but there were few further discoveries until 2004 when Rosas and Sagan \cite{RS} substantially advanced the area, discovering numerous noncommutative analogues of classical concepts in symmetric function theory, such as the RSK map and the $\omega$ involution. In particular, they gave noncommutative analogues of the bases of monomial, power sum, elementary, and complete homogeneous symmetric functions, each analogue mapping to its (scaled) symmetric function analogue under the projection map that lets the variables commute. 

\sagan{Rosas and Sagan did define \svwtwo{a partial set of} Schur functions in noncommuting variables, which projected to scaled classical Schur functions.  But they were not a basis for NCSym since they were indexed by integer, rather than set, partitions.} This led them to ask in \cite[Section 9]{RS} whether there is a way to define functions indexed by set partitions having properties analogous to the classical Schur functions. In answer, in 2006 Bergeron, Hohlweg, Rosas and Zabrocki \cite{BHRZ} suggested a possible candidate, the $\mathbf{x}$-basis, which corresponds to the classes of simple modules of the Grothedieck bialgebra of the semi-tower of partition lattice algebras, though to date no concrete connection to \svw{classical} Schur functions has been made. 

Since then much of the algebraic structure of $\NCSym$ has been revealed, including by Bergeron, Reutenauer, Rosas and Zabrocki   who introduced a natural Hopf algebra structure on $\NCSym$ \cite{BRRZ},  Lauve and Mastnak who then computed the antipode \cite{LauveM}, and Bergeron and Zabrocki who proved that $\NCSym$ was free and cofree \cite{BZ}. Connecting $\NCSym$ to combinatorics, Can and Sagan proved that $\NCSym$ was isomorphic to the algebra of rook placements \cite{CanSagan}, and the generalization of the chromatic symmetric function of Stanley \cite{Stan95} to $\NCSym$ by Gebhard and Sagan \cite{GebSag} enabled them to resolve cases of the celebrated \svwrev{$(\textbf{3}+\textbf{1})$-free} Conjecture of Stanley and Stembridge \cite{StanStem}. This generalization was also used by Dahlberg to resolve further cases of the \svwrev{$(\textbf{3}+\textbf{1})$-free} Conjecture of Stanley and Stembridge  \cite{Dladders}. Meanwhile in representation theory, $\NCSym$ arises in the supercharacter theory of all unipotent upper-triangular matrices over a finite field \cite{28authors, Thiem}.

However, the question of the existence of a basis of $\NCSym$ with properties analogous to \svw{classical} Schur functions, which maps to (scaled) Schur functions under the projection map remained open. This was despite the flourishing area of Schur-like functions pioneered by the quasisymmetric Schur functions of Haglund, Luoto, Mason and \svwrevtwo{van Willigenburg} \cite{QS}, and followed by discoveries of row-strict quasisymmetric Schur functions \cite{mason-remmel}, Young quasisymmetric Schur functions \cite{LMvW}, noncommutative Schur functions \cite{BLvW} and immaculate functions \cite{BBSSZ, XXL},  quasisymmetric Schur $Q$-functions \cite{jingli},  quasi-Grothendieck polynomials \cite{monical} and quasisymmetric Macdonald polynomials \cite{Corteeletal}. 

In this paper we resolve the question of Rosas and Sagan by introducing Schur functions in noncommuting variables, which map to the classical Schur functions under the projection map, yield a basis for $\NCSym$ indexed by set partitions, exhibit many properties analogous to classical Schur functions, and furthermore naturally refine the Rosas-Sagan Schur functions. More precisely, our paper is structured as follows.

In Section~\ref{sec:background} we recall the relevant background. Then in Section~\ref{sec:source} we define the genesis for our Schur functions in $\NCSym$, the source Schur functions, in Definition~\ref{def:sourceskew} and show they satisfy an analogue of the \svw{classical} Schur function product rule in Theorem~\ref{the:prod} that is later generalized in Proposition~\ref{prop:prod}. Generalizing these functions in Section~\ref{sec:schur}, we define our skew Schur functions in noncommuting variables in Definition~\ref{def:ncschurs} and use this to identify the Schur basis of $\NCSym$ in Theorem~\ref{the:ncschur}, which leads to further Schur-like bases through transpose and permutation actions in Theorem~\ref{the:transpose} and Corollary~\ref{cor:permbases}. In each of these cases we discover that our functions map naturally to (skew) Schur functions when we let the variables commute. Moreover we relate our functions to Specht modules in Subsection~\ref{subsec:tabloid}. \svw{In Section~\ref{sec:RS}} we prove that our (skew) Schur functions in noncommuting variables naturally refine the Rosas-Sagan (skew) Schur functions in Theorem~\ref{the:RSrefines} using the \FA{Lindstr\"om-Gessel-Viennot} swap, and prove noncommutative analogues of the Littlewood-Richardson rule and coproduct rule for them in Propositions~\ref{prop:RSLR} and \ref{prop:RScoprod}. \FA{Lastly, in Section~\ref{sec:NSym} we show our functions naturally extend immaculate functions indexed by integer partitons and noncommutative ribbon Schur functions.}

\section{Background}\label{sec:background}  
We begin by recalling the combinatorics and algebra we will need before proving some small but useful lemmas.

Given a positive integer $n$, we say an \emph{integer partition} $\lambda = \lambda _1 \lambda _2 \cdots \lambda _{\ell(\lambda)}$ of $n$ is a \svwrev{weakly decreasing} list of positive integers whose sum is $n$. We denote this by $\lambda \vdash n$, call the $\lambda _i$ the \emph{parts} of $\lambda$, $\ell(\lambda)$ the \emph{length} of $\lambda$ and $n$ the \emph{size} of $\lambda$. For convenience \svwrev{we denote} by $\emptyset$ the unique empty \svw{integer} \svwrevtwo{partition of}  0. If $i$ appears in $\lambda$ exactly $r_i$ times for $1\leq i \leq n$, then we can also write $\lambda = 1^{r_1}2^{r_2} \cdots n^{r_n}$. With this in mind we define
$$\lambda ! = \lambda _1 !\lambda _2 !\cdots \lambda _{\ell(\lambda)}! \qquad \lambda^! = {r_1}!{r_2}! \cdots  {r_n}!$$and the \emph{transpose} of $\lambda$ to be the integer partition
$$\lambda ^t = (r_1+r_2+\cdots + r_n)(r_2+\cdots + r_n)\cdots (r_n)$$with zeros removed.

\begin{example}\label{ex:partitions}
If $\lambda = 3221 = 1^12^23^14^05^06^07^08^0 \vdash 8$, then $\ell (\lambda) = 4$, $\lambda ! = 3!2!2!1! = 24$, $\lambda ^! = 1!2!1!0!0!0!0!0! = 2$ and $\lambda ^t = (1+2+1)(2+1)(1)=431$.
\end{example}

If $\lambda$ is an integer partition, then we say the \emph{diagram} of $\lambda$, also denoted by $\lambda$, is the array of left-justified boxes with $\lambda _i$ boxes in row $i$ from the top. Consider two integer partitions $\lambda \vdash n$ and \svw{$\mu \vdash m$} such that $\ell(\mu) \leq \ell(\lambda)$ and the parts satisfy $\mu_i \leq \lambda _i$ for all $1\leq i \leq \mu _{\ell(\mu)}$, which we denote by \svwtwo{$\mu \subseteq \lambda$.} We say that the \emph{skew diagram} $\lambda /\mu$ of \emph{size} $(n-m)$ is the array of boxes contained in $\lambda$ but not in $\mu$ when the array of boxes of $\mu$ \svwrevtwo{is} located in the top left corner of the array of boxes of $\lambda$. Note that $\lambda /\emptyset = \lambda$. Furthermore the \emph{concatenation} of $\lambda$ and $\mu$, denoted by $\lambda\cdot \mu$, is obtained from $\lambda$ and $\mu$ by aligning the rightmost column of $\mu$ immediately below the leftmost column of $\lambda$ . Similarly, the \emph{near concatenation} of $\lambda$ and $\mu$, denoted by $\lambda \odot \mu$, is obtained from $\lambda$ and $\mu$ by aligning the topmost row of $\mu$ immediately left of the bottommost row of $\lambda$.

\begin{example}\label{ex:skewdiags}
If $\lambda = 3221$ and $\mu = 21$, then their respective diagrams and skew diagram are
$$\lambda = \tableau{\ &\ &\ \\\ &\ \\\ &\ \\\  }\qquad \mu= \tableau{\ &\ \\\ }\qquad \lambda/\mu = \tableau{&&\ \\&\ \\\ &\ \\\  }$$where the latter is of size $8-3=5$. Furthermore their concatenation and near concatenation are, respectively, as follows.
$$\lambda \cdot \mu = \tableau{&\ &\ &\ \\&\ &\ \\&\ &\ \\&\ \\\  &\ \\\ } \qquad \lambda \odot \mu = 
\tableau{
&&\ &\ &\ \\
&&\ &\ \\
&&\ &\ \\
\ &\  &\ \\\ }$$
\end{example}

Given $[n] = \{1,2,\ldots , n\}$, we say a \emph{set partition} $\pi$ of $[n]$ is a family of disjoint nonempty subsets $B_1, B_2, \ldots, B_{\ell(\pi)}$ whose union is $[n]$. We denote this by
$$\pi = B_1/B_2/ \cdots /B_{\ell(\pi)} \vdash [n]$$call the $B_i$ the \emph{blocks} of $\pi$, $\ell(\pi)$ the \emph{length} of $\pi$ and $n$ the \emph{size} of $\pi$. For convenience we usually list the blocks by increasing least element and omit the set parentheses and commas of the blocks, \svwrev{and denote by $\emptyset$ the unique empty set partition of $[0]=\emptyset$.} For a finite set of integers $S$ we define $S+n = \{s+n : s\in S\}$. Given two set partitions $\pi \vdash [n]$ and $\sigma=  B_1/B_2/ \cdots /B_{\ell(\sigma)} \vdash [m]$ we say that their \emph{slash product} $\pi \slashp \sigma$ is
$$\pi \slashp \sigma = \pi/ B_1 +n/B_2 +n/ \cdots /B_{\ell(\sigma)} +n \vdash [n+m].$$

\begin{example}\label{ex:setpartitions}
If we take  the family of disjoint  sets $\{1,3,4\}, \{2,5\}, \{6\}, \{7,8\}$, then as a set partition we write
$$\pi = 134/25/6/78 \vdash [8].$$
If $\pi = 134/25 \vdash [5]$ and $\sigma = 1/23 \vdash [3]$, then $$\pi \slashp \sigma = 134/25/6/78 \vdash [8]$$and $\ell(\pi\slashp \sigma) =4.$
\end{example}

Observe that every set partition $\pi = B_1/B_2/\cdots / B_{\ell(\pi)}\vdash [n]$ determines a unique integer partition $\lambda (\pi)\vdash n$ via
$$\lambda (\pi) =  |B_1| |B_2| \cdots  |B_{\ell(\pi)}|$$with parts listed in weakly decreasing order. Conversely, every $\lambda = \lambda _1 \lambda _2 \cdots \lambda _{\ell(\lambda)} \vdash n$ determines a natural set partition $[\lambda] \vdash [n]$ via
\begin{align*}[\lambda] &= 12\cdots \lambda _1/(\lambda _1 +1) \cdots (\lambda_1+\lambda _2) / \cdots / \left(\sum _{i=1}^{\ell(\lambda)-1}\lambda _i \right)+1\cdots n\\
&=[\lambda _1] \slashp [\lambda _2] \slashp \cdots \slashp [\lambda _{\ell(\lambda)}].
\end{align*}

\begin{example}\label{ex:lambdatopi}
We have that
$$\lambda(134/25/6/78)=3221$$and
\begin{align*}[3221]&= 123/45/67/8=123\slashp 12 \slashp 12 \slashp 1=[3]\slashp [2] \slashp [2]\slashp [1].
\end{align*}
\end{example}

We now turn our attention to  \FA{two Hopf algebras} that will be of interest to us. The first of these is the graded \emph{Hopf algebra of symmetric functions}, \svw{$\Sym$,}
$$\Sym = \Sym ^0 \oplus \Sym ^1 \oplus \cdots \subset \bQ [[x_1, x_2, \ldots ]]$$where \sagan{$[[\cdot ]]$ means that the variables commute,} $\Sym ^0 = \spam \{1\}$ and the $n$th graded piece for $n\geq 1$ has the following renowned bases
\svwtwo{$$\begin{array}{rclcl}
\Sym ^n&=& \spam\{ m_\lambda \suchthat \lambda\vdash n\} &=& \spam\{ p_\lambda \suchthat \lambda\vdash n\}\\
&=& \spam\{ e_\lambda \suchthat \lambda\vdash n\} &=& \spam\{  h_\lambda \suchthat \lambda\vdash n\}\\
&=& \spam\{ s_\lambda \suchthat \lambda\vdash n\} &&
\end{array}$$}where these functions are defined as follows, \svw{given an integer partition} $\lambda = \lambda_1\lambda_2\cdots \lambda_{\ell(\lambda)}\vdash n$.

The \emph{monomial symmetric function}, $m_\lambda$, is given by
$$m_\lambda = \sum x_{i_1}^{\lambda _1}x_{i_2}^{\lambda _2}\cdots x_{i_{\ell(\lambda)}}^{\lambda _{\ell(\lambda)}}$$summed over distinct monomials and the $i_j$ are also distinct.

\begin{example}\label{ex:monomials}
$m_{21} = x_1^2x_2 + x_2^2x_1+ \cdots$
\end{example}

The \emph{power sum symmetric function}, $p_\lambda$, is given by
$$p_\lambda = p_{\lambda _1}p_{\lambda _2}\cdots p_{\lambda _{\ell(\lambda)}}$$where \svwtwo{$p_{i} = x_1^{i}+x_2^{i}+\cdots$.}

\begin{example}\label{ex:powers}
$p_{21} = (x_1^2+x_2^2+ \cdots)(x_1+x_2+\cdots)$
\end{example}

The \emph{elementary symmetric function}, $e_\lambda$, is given by
$$e_\lambda = e_{\lambda _1}e_{\lambda _2}\cdots e_{\lambda _{\ell(\lambda)}}$$where \svwtwo{$e_{i} = \sum_{j_1<j_2<\cdots <j_{i}} x_{j_1}x_{j_2}\cdots x_{j_{i}}$.}

\begin{example}\label{ex:elementaries}
$e_{21} = (x_1x_2+x_2x_3+ \cdots)(x_1+x_2+\cdots)$
\end{example}

The \emph{complete homogeneous symmetric function}, $h_\lambda$, is given by
$$h_\lambda = h_{\lambda _1}h_{\lambda _2}\cdots h_{\lambda _{\ell(\lambda)}}$$where \svwtwo{$h_{i} = \sum_{j_1\leq j_2\leq\cdots \leq j_{i}} x_{j_1}x_{j_2}\cdots x_{j_{i}}$.}

\begin{example}\label{ex:completes}
$h_{21} = (x_1x_2+x_1^2+ \cdots)(x_1+x_2+\cdots)$
\end{example}

Our final basis consists of \SXL{Schur} functions, and we use the complete homogeneous symmetric functions and elementary symmetric functions to give two ways to define them, which we will use later. The \emph{Schur function}, $s_\lambda$, is given by the \emph{Jacobi-Trudi determinant} and its dual:
\begin{equation}\label{eq:JT}
s_\lambda = \det \left( h_{\lambda _i  - i + j} \right) _{1\leq i,j \leq \ell(\lambda)} = \det \left( e_{(\lambda^t) _i - i + j} \right) _{1\leq i,j \leq \ell(\lambda^t)}
\end{equation}where $h_0=e_0=1$ and any function with a negative index equals $0$. Given a skew diagram $\lambda / \mu$ we similarly define the \emph{skew Schur function} $s_{\lambda/\mu}$ to be
\begin{equation}\label{eq:JTskew}
s_{\lambda/\mu} = \det \left( h_{\lambda _i  -\mu_j - i + j} \right) _{1\leq i,j \leq \ell(\lambda)} = \det \left( e_{(\lambda^t) _i -(\mu^t)_j - i + j} \right) _{1\leq i,j \leq \ell(\lambda^t)}
\end{equation}where we set $\mu_j = 0$ for $\ell(\mu) <j \leq \ell(\lambda)$.

\begin{example}\label{ex:schur}
$s_{21}= \det \begin{pmatrix}h_2&h_3\\h_0&h_1\end{pmatrix} = h_{21} -h_3\qquad s_{22/1}= \det \begin{pmatrix}h_1&h_3\\h_0&h_2\end{pmatrix} = h_{12}-h_3 = h_{21} -h_3$
\end{example}

We now turn our attention to the second Hopf algebra of interest to us, the graded \emph{Hopf algebra of symmetric functions in noncommuting variables}, $\NCSym$,
$$\NCSym  = \NCSym ^0 \oplus \NCSym ^1 \oplus \cdots \subset \bQ \langle \langle x_1, x_2, \ldots \rangle\rangle$$where \sagan{$\langle \langle\cdot \rangle\rangle$ means that the variables do not commute,} $\NCSym ^0 = \spam \{1\}$ and {the} $n$th graded piece for $n\geq 1$ has the following bases \svw{\cite{RS}}, known respectively as the $n$th graded piece of the \emph{$m$-, $p$-, $e$-, $h$-basis of $\NCSym$},
\svwtwo{$$\begin{array}{rclcl}
\NCSym ^n&=& \spam\{ m_\pi \suchthat \pi\vdash [n]\} &=& \spam\{ p_\pi \suchthat \pi\vdash [n]\}\\
&=& \spam\{ e_\pi \suchthat \pi\vdash [n]\} &=& \spam\{  {h_\pi} \suchthat \pi\vdash [n]\}
\end{array}$$}where these functions are defined as follows, given a set partition $\pi \vdash [n]$.

The \emph{monomial symmetric function in $\NCSym$}, $m_\pi$, is given by 
$$m_\pi = \sum _{(i_1, i_2, \ldots , i_n)} x_{i_1}x_{i_2} \cdots x_{i_n}$$summed over all tuples $(i_1, i_2, \ldots , i_n)$ with $i_j=i_k$ if and only if $j$ and $k$ are in the same block of $\pi$.

\begin{example}\label{ex:mpi}
$m_{13/2}=x_1x_2x_1+x_2x_1x_2+x_1x_3x_1+x_3x_1x_3+x_2x_3x_2+x_3x_2x_3+\cdots$
\end{example}

\sagan{For the next two definitions, the implication only goes one way.}

The \emph{power sum symmetric function in $\NCSym$}, $p_\pi$, is given by 
$$p_\pi = \sum _{(i_1, i_2, \ldots , i_n)} x_{i_1}x_{i_2} \cdots x_{i_n}$$summed over all tuples $(i_1, i_2, \ldots , i_n)$ with $i_j=i_k$ if  $j$ and $k$ are in the same block of $\pi$.

\begin{example}\label{ex:ppi}
$p_{13/2}=x_1x_2x_1+x_2x_1x_2+\cdots + x_1^3+x_2^3 +\cdots $
\end{example}

The \emph{elementary symmetric function in $\NCSym$}, $e_\pi$, is given by 
$$e_\pi = \sum _{(i_1, i_2, \ldots , i_n)} x_{i_1}x_{i_2} \cdots x_{i_n}$$summed over all tuples $(i_1, i_2, \ldots , i_n)$ with $i_j\neq i_k$ if  $j$ and $k$ are in the same block of $\pi$.

\begin{example}\label{ex:epi}
$e_{13/2}= {x_1x_1x_2+x_1x_2x_2+x_2x_2x_1+x_2x_1x_1}+\cdots + x_1x_2x_3+x_2x_3x_4 +\cdots$
\end{example}

The definition for the complete homogeneous symmetric functions in $\NCSym$ requires a few more constituents. Define an order on the set partitions of $[n]$ by refinement, that is, for two
set partitions $\pi,\sigma\vdash [n]$, we say that $\pi\leq \sigma$ if every block of $\pi$ is contained in some block of $\sigma$. Then \sagan{this gives a lattice and} the greatest
lower bound operation is denoted by $\wedge$.
With this in mind, the \svw{\emph{complete homogeneous symmetric function in $\NCSym$}}, $h_\pi$, is given by 
$$h_\pi=\sum_{\sigma} \lambda (\sigma \wedge \pi)! m_{\sigma}.$$ 

\begin{example}\label{ex:hpi}
$h_{13/2}= 2 m_{123} + m_{12/3} + m_{1/23} + 2 m_{13/2} + m_{1/2/3}$
\end{example}

However, the following new interpretation will be more useful to us, \svwtwo{where $\Sn _n$ is the symmetric group of size $n$ and $\mathbb{N}= \{1, 2, \ldots \}$.}

\begin{lemma}\label{lem:h}
Let $\pi=\pi_1/\pi_2/\cdots/\pi_{\ell(\pi)}\vdash [n]$. Then 
\begin{align}{\label{h-mon}}
h_{\pi}=\sum_{\eta} \sum_{(i_1,i_2,\ldots,i_n)  } x_{i_{\eta(1)}}x_{i_{\eta(2)}}\cdots x_{i_{\eta(n)}}
\end{align} 
where
\begin{enumerate}[1)]
\item   the first sum is over all $\eta\in \Sn_n$ that fixes the blocks of $\pi$,
\item   the second sum is over all $(i_1,i_2,\ldots,i_n)\in \mathbb{N}^n$ such that if $j$ and $k$ are in the same block of $\pi$ with $j<k$, then $i_j\leq i_k$.
\end{enumerate} 
\end{lemma} 
\begin{proof}
Consider  $\sigma=\sigma_1/\sigma_2/\cdots/\sigma_{\ell(\sigma)}\vdash [n]$. Fix an arbitrary term $x^{\sigma}=x_{i_1}x_{i_2}\cdots x_{i_n}$ where $i_j=i_k$ if and only if $j$ and $k$ are in the same block of $\sigma$. We write $i_j=a(p)$ if $j\in\sigma_p$, \svwrev{namely, $a(p)$ is the subscript of all variables appearing in the positions indexed by block $\sigma _p$.}

\svwtwo{Note that the coefficient of $m_\sigma$ in the expansion of $h_\pi$ in terms of the $m$-basis of $\NCSym$, $\lambda(\sigma \wedge \pi)!$,  is equal to the coefficient of $x^{\sigma}$ in the polynomial expansion of $h_{\pi}$ by the original definition of $h_{\pi}$ above. Therefore, if we show that the coefficient of $x^{\sigma}$ in 
Equation~\eqref{h-mon} is equal to $\lambda(\sigma \wedge \pi)!$, \svwrevtwo{then} we are done.}

\svwtwo{
In Equation~\eqref{h-mon}, the number of $x_{i_{\eta(1)}}x_{i_{\eta(2)}}\cdots x_{i_{\eta(n)}}$ that are equal to $x^{\sigma}$ is equal to the number of pairs $$(\eta,(i_1,i_2,\ldots,i_n))$$where
\begin{enumerate}[1)]
\item $\eta\in \Sn_n$ fixes the blocks of $\pi$, 
\item  $(i_1,i_2,\ldots,i_n)\in \mathbb{N}^n$ such that if $j$ and $k$ are in the same block of $\pi$ with $j<k$, then $i_j\leq i_k$, and
\item $i_{\eta(j)}=a(p)$ if and only if $\eta(j)\in \sigma_{p}$. 
\end{enumerate}
Note that only one tuple satisfies 3), which is $(i_{\eta(1)},i_{\eta(2)},\ldots,i_{\eta(n)})$  where $ i_{\eta(j)}=a(p)$ if and only if $\eta(j)\in \sigma_{p}$. Then the number of viable $\eta$ can be calculated as follows. For each permutation that fixes the blocks of $\pi$, it is a viable $\eta$ if it further only permutes all $ i_{\eta(j)}=a(p)$ in a given block of $\pi$ amongst themselves. That is, $\eta$ is viable if and only if $\eta$ fixes the blocks of $\pi \wedge \sigma$. Then, the number of viable $\eta$, subject to the ordering in 2), is $\lambda(\pi \wedge \sigma)!= \lambda(\sigma \wedge \pi)!$, and our result follows.}
\end{proof} 

Relating $\NCSym$ to $\Sym$ is the \emph{projection map} 
$$\rho \suchthat \bQ \langle \langle x_1, x_2, \ldots \rangle \rangle \rightarrow \bQ [[ x_1, x_2, \ldots ]]$$that simply lets the variables commute and yields the following result.

\begin{lemma}\label{lem:RSrho} \cite[Theorem 2.1]{RS} Let $\pi$ be a set partition.
\begin{enumerate}[1)]
\item $\rho(m_\pi) = \lambda (\pi) ^! m_{\lambda(\pi)}$
\item $\rho(p_\pi) = p_{\lambda(\pi)}$
\item $\rho(e_\pi) = \lambda (\pi)! e_{\lambda(\pi)}$
\item $\rho(h_\pi) = \lambda (\pi)! h_{\lambda(\pi)}$
\end{enumerate}
\end{lemma}

There are two other maps that will be useful to us. \svwtwo{The first is the \emph{permutation map} \cite[p. 219]{RS} and \cite[p. 230]{GebSag}, which is an action on places (not variables) as follows. Given $\delta \in \Sn _n$ and a monomial of degree $n$ in noncommuting variables, define
$$\delta \circ (x_{i_1}x_{i_2} \cdots x_{i_n}) = x_{i_{\delta^{-1}(1)}}x_{i_{\delta^{-1}(2)}} \cdots x_{i_{\delta^{-1}(n)}}$$and extend linearly. In \cite[p. 219]{RS} they also noted that if $\pi$ is a basis element of any of the above bases of $\NCSym$ and $\delta$ is a permutation, then
\begin{equation}\label{eq:delta}
\delta \circ b_\pi = b_{\delta\pi}
\end{equation}where $\delta$ acts on set partitions in the natural way.} The second is the classical involution $\omega$ defined on $\Sym$, whose analogue in $\NCSym$ is also an involution denoted by $\omega$ \cite[p. 221]{RS} defined by
\begin{equation}\label{eq:omega}
\omega(h_\lambda)=e_\lambda \qquad \omega(h_\pi) = e_\pi.
\end{equation}

\svwtwo{We are now ready for some preliminary lemmas.}



\begin{lemma}\label{lem:omegaonp}
For set partitions $\pi$ and $\sigma$ we have that
$$\omega (p_\pi p _\sigma) = \omega(p_\pi)\omega(p_\sigma).$$
\end{lemma}

\begin{proof}
Our result follows since
\svw{$$\omega (p_\pi p _\sigma) = \omega (p _{\pi \slashp \sigma}) = (-1)^{\svwrev{\pi\slashp\sigma}}p_{\pi \slashp \sigma} = (-1)^{\pi} p_{\pi} (-1)^{\sigma}p_{\sigma} =  \omega(p_\pi)\omega(p_\sigma)$$}where the first equality follows from \cite[Lemma 4.1 (i)]{BHRZ}, which states that $p_\pi p _\sigma = p_{\pi \slashp \sigma}$, and the second and fourth equalities follow from \cite[Theorem 3.5 (ii)]{RS}, which states that \svw{$\omega (p_\pi) = (-1)^\pi p_{\pi}$} where $(-1)^\pi$ is the sign of any permutation obtained \svwtwo{by replacing} each block of $\pi$ by a cycle.
\end{proof}

\begin{lemma}\label{lem:omegadeltacommute}
For a set partition $\pi \vdash [n]$ and permutation $\delta \in \Sn _n$ we have that
$$\omega \delta (p_\pi) = \delta \omega (p_\pi).$$
\end{lemma}

\begin{proof} By Equation~\eqref{eq:delta} and \cite[Theorem 3.5 (ii)]{RS} described in the previous proof, we have that
$$\omega \delta (p_\pi) = \omega (p_{\delta\pi}) = (-1) ^\pi p_{\delta\pi} = \delta ((-1) ^\pi p_{\pi}) = \delta \omega (p_\pi).$$
\end{proof}

\begin{lemma}\label{lem:deltarhocommute} For a set partition $\pi \vdash [n]$ and permutation $\delta \in \Sn _n$ we have that
\SXL{$$\rho\delta (p_\pi) = \rho (p_\pi).$$}
\end{lemma}

\begin{proof} 
By Equation~\eqref{eq:delta} and Lemma~\ref{lem:RSrho}, we have that
\SXL{$$\rho \delta (p_\pi) = \rho (p_{\delta\pi}) =  p_{\lambda(\pi)} =  \rho (p_\pi)$$}since \svw{the action of a permutation on a set partition maintains the block sizes.}\end{proof}

While the power sum functions in  $\NCSym$ have been useful for establishing key lemmas for $\NCSym$, our main focus will be the complete homogeneous symmetric functions in $\NCSym$, and we \svwrev{recall} a critical result for them next in analogy to \cite[Lemma 4.1]{BHRZ}. \svw{We will use it frequently and hence without citation whenever we multiply complete homogeneous symmetric functions in $\NCSym$.}

\begin{lemma}\cite[Corollary 2.41]{ABB}\label{lem:hmult}
For set partitions $\pi$ and $\sigma$ we have that
$$h_\pi h_\sigma = h_{\pi\slashp\sigma}.$$
\end{lemma}

One last definition we need is a \emph{noncommutative analogue of Leibniz' determinantal formula} for any matrix $A=(a_{ij}) _{1\leq i,j\leq n}$ with noncommuting entries $a_{ij}$, which we define to be
\begin{equation}\label{eq:leibniz}
\bdet (A) = \sum _{\varepsilon \in \Sn _n} \mathrm{sgn} (\varepsilon) a_{1\varepsilon (1)}a_{2\varepsilon (2)}\cdots a_{n\varepsilon (n)}
\end{equation}
that takes the product of the entries from the top row to the bottom row, and $\mathrm{sgn} (\varepsilon)$ is the sign of permutation $\varepsilon$.

\section{Source Schur functions}\label{sec:source}  
Before we define our Schur functions in noncommuting variables, we will define functions that will be their genesis.

\begin{definition}\label{def:sourceskew}
Let $\lambda / \mu$ be a skew diagram. Then the \emph{source skew Schur function in noncommuting variables} $s_{[\lambda/\mu]}$ is defined to be
\begin{equation}
\label{eq:skewncschur}
s_{[\lambda/\mu]} = \bdet \left( \frac{1}{(\lambda _i -\mu _j - i +j)! } h_{[\lambda _i -\mu_j - i + j]}\right) _{1\leq i,j \leq \ell(\lambda)}\end{equation}where we set $\mu_j = 0$ for all $\ell(\mu)< j \leq \ell(\lambda)$, \svw{$h_{[0]}=h_\emptyset = 1$ and any function with a negative index equals 0.} When $\mu = \emptyset$, we call $s_{[\lambda]}$ a \emph{source Schur function in noncommuting variables}.
\end{definition}

\begin{example}\label{ex:sourceskew} The source  Schur function in noncommuting variables $s_{[21]}$ is
\begin{align*}s_{[21]} &  = \svwrev{\bdet \begin{pmatrix} \frac{1}{2!} h_{[2]}& \frac{1}{3!} h_{[3]}\\
\frac{1}{0!} h_{[0]}& \frac{1}{1!} h_{[1]}
\end{pmatrix}} = \bdet \begin{pmatrix} \frac{1}{2!} h_{12}& \frac{1}{3!} h_{123}\\
\frac{1}{0!} h_{\emptyset}& \frac{1}{1!} h_{1}
\end{pmatrix}\\
&= \frac{1}{2!} h_{12} \frac{1}{1!} h_{1} - \frac{1}{3!} h_{123}\frac{1}{0!} h_{\emptyset} = \svw{\frac{1}{2} h_{12 \slashp 1} - \frac{1}{6} h_{123}} = \frac{1}{2} h_{12/3} - \frac{1}{6} h_{123}.
\end{align*}
Meanwhile the source skew Schur function in noncommuting variables $s_{[22/1]}$ is
\begin{align*}s_{[22/1]} &=  \svwrev{\bdet \begin{pmatrix} \frac{1}{1!} h_{[1]}& \frac{1}{3!} h_{[3]}\\
\frac{1}{0!} h_{[0]}& \frac{1}{2!} h_{[2]}
\end{pmatrix}} =  \bdet \begin{pmatrix} \frac{1}{1!} h_{1}& \frac{1}{3!} h_{123}\\
\frac{1}{0!} h_{\emptyset}& \frac{1}{2!} h_{12}
\end{pmatrix}\\
&= \frac{1}{1!} h_{1}\frac{1}{2!} h_{12}  - \frac{1}{3!} h_{123}\frac{1}{0!} h_{\emptyset}  = \svw{\frac{1}{2} h_{1\slashp 12} - \frac{1}{6} h_{123}} = \frac{1}{2} h_{1/23} - \frac{1}{6} h_{123}.
\end{align*}
\end{example}

Our first result explains the use of the terminology skew Schur function, because the commutative image of our functions are the classical skew Schur functions.

\begin{proposition}\label{prop:rhoskew}
$$\rho (s_{[\lambda / \mu]}) = s_{\lambda / \mu}$$
\end{proposition}

\begin{proof} By definition,
\begin{align*}\rho(s_{[\lambda/\mu]})&= \rho  \left(\bdet \left( \frac{1}{(\lambda _i -\mu _j - i +j)! } h_{[\lambda _i -\mu _j - i + j]}\right) _{1\leq i,j \leq \ell(\lambda)} \right)\\
&=  \bdet \left( \frac{1}{(\lambda _i -\mu_j - i +j)! } \rho (h_{[\lambda _i -\mu_j - i + j]})\right) _{1\leq i,j \leq \ell(\lambda)}\\
&=  \det \left( h_{\lambda _i -\mu_j - i + j} \right) _{1\leq i,j \leq \ell(\lambda)}\\
&= s_{\lambda/\mu}\end{align*}
where the distributivity of $\rho$ \svwtwo{follows from the fact that $\rho$ is a Hopf morphism \cite[Proposition 4.3]{BRRZ},} the penultimate equality follows from Lemma~\ref{lem:RSrho}, and the last equality is by the definition of skew Schur functions in Equation~\eqref{eq:JTskew}.
\end{proof}

Furthermore, source Schur functions in $\NCSym$ satisfy a product rule analogous to that of Schur functions in $\Sym$, namely
\begin{equation}\label{eq:classicprod}
s_{\lambda}s_{\mu}= s_{\lambda \cdot \mu}+ s_{\lambda \odot \mu}.
\end{equation}

\begin{theorem}\label{the:prod}
$$s_{[\lambda]}s_{[\mu]}= s_{[\lambda \cdot \mu]}+ s_{[\lambda \odot \mu]}$$
\end{theorem}

\begin{proof} For ease of notation, for a skew diagram $D$, we will denote the matrix in Definition~\ref{def:sourceskew} of $s_{[D]}$ by $\JT_{[D]}$. 

We will prove the equivalent identity
$$s_{[\lambda \cdot \mu]}=s_{[\lambda]}s_{[\mu]}-s_{[\lambda \odot \mu]}.$$
By definition 
\begin{align*}s_{[\lambda \cdot \mu]}&= \bdet (\JT_{[\lambda \cdot \mu]})\\
&= \bdet \svwrevtwo{\begin{pmatrix}
\frac{h _{[\lambda _1]}}{\lambda _1 !}& \frac{h _{[\lambda _1+1]}}{(\lambda _1 +1)!}&\cdots &\vdots&&\\
\frac{h _{[\lambda _2 -1]}}{(\lambda _2-1) !}& \frac{h _{[\lambda _2]}}{\lambda _2 !}&\cdots &\vdots&&\iddots\\
\vdots&\vdots&\ddots&\vdots\\
\cdots&\cdots&\cdots&\frac{h _{[\lambda_{\ell(\lambda)}]}}{\lambda_{\ell(\lambda)}!}&\frac{h_{[{\lambda_{\ell(\lambda)}+\mu_1}]}}{({\lambda_{\ell(\lambda)}+\mu_1})!}&\cdots&\cdots&\cdots\\
0 &\cdots&0&1&\frac{h _{[\mu _1]}}{\mu _1 !}&\frac{h _{[\mu _1+1]}}{(\mu _1 +1)!}&\cdots&\cdots\\
0 &\cdots&0&0&\frac{h _{[\mu _2 -1]}}{(\mu_2-1) !}& \frac{h _{[\mu _2]}}{\mu _2 !}\\
\vdots &\cdots&\vdots&\vdots&\vdots&&\ddots\\
0 &\cdots&0&0&\vdots&&&\frac{h _{[\mu_{\ell(\mu)}]}}{\mu_{\ell(\mu)!}}\\
\end{pmatrix}}
\end{align*} and note that we have the block decomposition
$$\JT_{[\lambda\cdot\mu]}=\begin{pmatrix}\mathrm{JT}_{[\lambda]}&A\\Z&\mathrm{JT}_{[\mu]}\end{pmatrix}
$$
where $Z$ is entirely zeros except for a 1 in the upper-right corner. So every term in the determinant comes either from the diagonal matrix $\mathrm{diag}(\JT_{[\lambda]}, \JT_{[\mu]})$ or contains the 1 of $Z$. The terms of the first kind give $s_{[\lambda]} s_{[\mu]}$. For those of the second kind, the position of the 1 gives a minus sign, and then removing that row and column gives the matrix $\JT_{[\lambda\odot\mu]}$ for the other term. Hence, the result follows.\end{proof}

\begin{example}\label{ex:prod}
The following two examples demonstrate our product.
$$s_{[1]}s_{[21]}= s_{[221/1]} + s_{[31]} \qquad s_{[21]}s_{[1]}= s_{[211]} + s_{[32/1]}$$
\end{example}

\section{Schur functions in $\NCSym$}\label{sec:schur}  
\subsection{The standard  and permuted bases}\label{subsec:standard} 
We now use our source Schur functions to define Schur functions in noncommuting variables, but first we will require tableaux that will be useful to us in a variety of ways later. 

Consider a skew diagram $\lambda/\mu$ of size $n$. We say that $t$ is a \emph{Young tableau} of \emph{shape} $\sh (t)=\lambda/\mu$ if the boxes of $\lambda/\mu$ are filled bijectively with $1, 2, \ldots, n$. The \emph{permutation} of $t$ in one-line notation, denoted by $\delta _t\in \Sn _n$, is obtained by reading the entries of each row of $t$ from left to right, and reading the rows of $t$ from top to bottom. \svwrevtwo{Consequently, we have the correspondence
$$t \leftrightarrow (\delta_t, \sh (t)).$$}

Now consider the set of all Young tableaux $t$ such that
\begin{enumerate}[1)]
\item $\sh (t) = \lambda$ for some integer partition $\lambda\vdash n$,
\item the entries in each row of $t$ increase from left to right,
\item if $\lambda = \lambda _1 \lambda _2\cdots \lambda _{\ell(\lambda)}$ and $\lambda _i = \lambda _j$ with $i<j$, then in $t$
$$\mbox{(the leftmost entry of row $i$)}  < \mbox{(the leftmost entry of row $j$)}.$$
\end{enumerate}
Observe that this set is in bijection with the set consisting of all set partitions of $[n]$:  Young tableau \svwrevtwo{$t _\pi$} corresponds to  set partition $\pi$ if and only if the set of entries for each row of \svwrevtwo{$t_\pi$} are precisely the blocks of $\pi$. In this case we define $$\svwrevtwo{\delta _\pi = \delta _{t_\pi}}$$and note that \svwrevtwo{$\lambda(\pi) = \sh (t _\pi)$.} \svwrevtwo{Consequently, we have the correspondence
$$\pi \leftrightarrow t _\pi \leftrightarrow (\delta_{\pi}, \sh (t_\pi)).$$}

\begin{example}\label{ex:tabs}
If $t$ is
$$\tableau{&&3&8&7\\&&2\\1&9&6\\5&4}$$then $\delta _t = 387219654$.
If $t_\pi$ is
$$\tableau{1&6&9\\3&7&8\\4&5\\2}$$then $\pi = 169/378/45/2 = 169/2/378/45$ and  $\delta _\pi = \delta _{t_\pi} = 169378452$.
\end{example}

We are now ready to define skew Schur functions and Schur functions in noncommuting variables.

\begin{definition}\label{def:ncschurs}
Let $\lambda/\mu$ be a skew diagram of size $n$ and $\delta \in \Sn _n$. Then the \emph{skew Schur function in noncommuting variables} $s_{(\delta, \lambda /\mu)}$ is defined to be
\begin{equation}\label{eq:skewNCSchur}
s_{(\delta, \lambda /\mu)} = \delta \circ s_{[\lambda/\mu]} = \delta \circ \bdet \left( \frac{1}{(\lambda _i -\mu _j - i +j)! } h_{[\lambda _i -\mu_j - i + j]}\right) _{1\leq i,j \leq \ell(\lambda)}.
\end{equation}
Moreover, if $\mu = \emptyset$, then we call $s_{(\delta, \lambda)}$ a \emph{Schur function in noncommuting variables}. 

Furthermore, if $\pi \vdash [n]$ and $\lambda (\pi) = \lambda _1 \lambda _2 \cdots \lambda _{\ell(\pi)}$, then the \emph{standard Schur function in noncommuting variables} $s_\pi$ is defined to be
\begin{equation}\label{eq:NCSchur}
s_{\pi} = s_{(\delta _\pi, \lambda (\pi))}= \delta _\pi \circ s_{[\lambda (\pi)]} = \delta _\pi \circ \bdet \left( \frac{1}{(\lambda   _i  - i +j)! } h_{[\lambda   _i  - i + j]}\right) _{1\leq i,j \leq \ell(\lambda(\pi))}.
\end{equation}
\end{definition}

\SXL{We remark that unlike the other bases in $\NCSym$, $\delta \circ s_{\pi} \neq s_{\delta \pi}$ in general, as we will see in Corollary~\ref{cor:permbases}.}

\begin{example}\label{ex:ncschurs}
If $\pi = 12/3$, then $t_\pi=\tableau{1&2\\3}$ and $\delta _\pi = 123 = \id$. Hence, \svwrev{by Example~\ref{ex:sourceskew},} the standard Schur function in noncommuting variables $s_{12/3}$ is
\begin{align*}s_{12/3} &= \id \circ s_{[21]} = \id \circ\bdet \begin{pmatrix} \frac{1}{2!} h_{12}& \frac{1}{3!} h_{123}\\
\frac{1}{0!} h_{\emptyset}& \frac{1}{1!} h_{1}
\end{pmatrix}\\
&= \frac{1}{2!} h_{12} \frac{1}{1!} h_{1} - \frac{1}{3!} h_{123}\frac{1}{0!} h_{\emptyset}  = \svw{\frac{1}{2} h_{12\slashp 1} - \frac{1}{6} h_{123}} = \frac{1}{2} h_{12/3} - \frac{1}{6} h_{123}. 
\end{align*}
If $\pi = 13/2$, then $t_\pi=\tableau{1&3\\2}$ and $\delta _\pi = 132$. Hence, \svwrev{by Example~\ref{ex:sourceskew},} the standard Schur function in noncommuting variables $s_{13/2}$ is
$$s_{13/2} = 132 \circ s_{[21]} = 132\circ \left(\frac{1}{2} h_{12/3} - \frac{1}{6} h_{123}\right) = \frac{1}{2} h_{13/2} - \frac{1}{6} h_{123}.$$
\end{example}

Our next lemma explains why our former functions are so called, and generalizes Proposition~\ref{prop:rhoskew}.

\begin{lemma}\label{lem:rhoschur}
$$\rho (s_{(\delta, \lambda / \mu)}) = s_{\lambda / \mu}$$
\end{lemma}

\begin{proof} \SXL{By Lemma~\ref{lem:deltarhocommute} we have that
$$\rho(s_{(\delta, \lambda / \mu)}) = \rho(\delta \circ s_{[\lambda/\mu]}) = \rho(s_{[\lambda/\mu]})=  s_{\lambda / \mu}$$
where the last equality follows by Proposition~\ref{prop:rhoskew}.}\end{proof}

Our next theorem explains why our latter functions are so called, since they are the natural analogues in $\NCSym$ indexed by set partitions, of the standard Schur functions in $\Sym$ indexed by integer partitions. We call the resulting basis the \emph{Schur basis} of $\NCSym$.

\begin{theorem}\label{the:ncschur}
The set  $\{ s_\pi \} _{\pi \vdash [n], \svwrev{n\geq 0}}$ is a basis for $\NCSym$. Moreover,
$$\rho (s_\pi) = s_{\lambda(\pi)}.$$
\end{theorem}

\begin{proof} The second part follows immediately by Lemma~\ref{lem:rhoschur} when $\delta = \delta _\pi$, $\lambda = \lambda(\pi)$ and $\mu=\emptyset$.

\svwrev{For the first part, let us consider the transition matrix from standard Schur functions in noncommuting variables to complete homogeneous symmetric functions in noncommuting variables, where the indexing set partitions are ordered by the lexicographic ordering of  \svwrevtwo{$\lambda(\pi)$} and ties are broken arbitrarily between set partitions $\pi$ and $\sigma$ satisfying $\lambda(\pi)= \lambda(\sigma)$. Observe by the definition of $s_\pi$  that $h _\pi$ appears in the expansion of $s_\pi$ with coefficient $\frac{1}{\lambda(\pi)!}$ and all other terms $h_\sigma$ appearing in the expansion of $s_\pi$ have $\sigma$ that satisfies \svwrevtwo{$\lambda(\sigma)$} is greater than \svwrevtwo{$\lambda(\pi)$} in lexicographic order. Hence $\sigma$ is larger than $\pi$ in our indexing order, and the result follows.}
\end{proof}

Another method to produce a basis for $\NCSym$ is to apply $\omega$ to each $s_\pi$, which we study next. The \emph{transposed standard  Schur function in noncommuting variables} for a set partition $\pi$ is defined to be
\begin{equation}\label{eq:ncschurt}
\st _\pi = \omega (s_\pi)
\end{equation}and our next theorem gives an explicit formula to compute these.

\begin{theorem}\label{the:transpose} Let $\pi$ be a set partition and $\lambda (\pi) = \lambda _1 \lambda _2 \cdots \lambda _{\ell(\pi)}$. Then \begin{equation}
\label{eq:ncschurtine}
\st _\pi = \delta _\pi \circ \bdet \left( \frac{1}{(\lambda _i - i +j)! } e_{[\lambda _i - i + j]}\right) _{1\leq i,j \leq \ell(\pi)}.
\end{equation}
The set  $\{ \st _\pi \} _{\pi \vdash [n], \svwrev{n\geq 0}}$ is a basis for $\NCSym$. Moreover,
$$\rho (\st _\pi) = s_{\lambda(\pi)^t}.$$
\end{theorem}

\begin{proof} For the first part note that
\begin{align*}
\st _\pi &= \omega (s_\pi) = \omega\delta _\pi \circ \bdet \left( \frac{1}{(\lambda _i - i +j)! } h_{[\lambda _i - i + j]}\right) _{1\leq i,j \leq \ell(\pi)}\\
&= \delta _\pi \circ \omega \bdet \left( \frac{1}{(\lambda _i - i +j)! } h_{[\lambda _i - i + j]}\right) _{1\leq i,j \leq \ell(\pi)}\\
&= \delta _\pi \circ  \bdet \left( \frac{1}{(\lambda _i - i +j)! } e_{[\lambda _i - i + j]}\right) _{1\leq i,j \leq \ell(\pi)}
\end{align*}where the first equality follows by definition, and the last two equalities follow by Lemma~\ref{lem:omegadeltacommute} and Lemma~\ref{lem:omegaonp} respectively. The second part now follows by Theorem~\ref{the:ncschur} and because $\omega$ is a bijection.

For the third part, since $\omega$ and $\rho$ commute \cite[Theorem 3.5 (iii)]{RS}
$$\rho (\st _\pi) = \rho\omega(s_\pi)=\omega\rho(s_\pi) = \omega (s_{\lambda(\pi)})=s_{\lambda(\pi)^t}$$by \svwrevtwo{Theorem~\ref{the:ncschur} and} the definition of $\omega$.\end{proof}

In contrast to $\Sym$,  the Schur basis given by standard Schur functions in noncommuting variables, and the basis given \svwrevtwo{by} transposed standard Schur functions in noncommuting variables are different.

\begin{proposition}\label{prop:2bases}
$$\{ s _\pi \} _{\pi \vdash [n], \svwrev{n\geq 0}}\neq \{ \st _\pi \} _{\pi \vdash [n], \svwrev{n\geq 0}}$$
\end{proposition}

\begin{proof} In order to prove this, it suffices to find an element $s_\sigma \in \{ s _\pi \} _{\pi \vdash [n], \svwrev{n\geq 0}}$ such that $s_\sigma \not \in \{ \st _\pi \} _{\pi \vdash [n], \svwrev{n\geq 0}}$. Let us consider $\sigma = [n]$ for some $n \geq 3$. Then if we assume that $s_{[n]} \in \{ \st _\pi \} _{\pi \vdash [n], \svwrev{n\geq 0}}$, by definition we must have that $\omega (s_\pi) = s_{[n]}$ for some $\pi$. By   applying the involution $\omega$ to both sides we obtain that
$$s_\pi = \omega(s_{[n]}) = \frac{1}{n!} e_{[n]} = \frac{1}{n!} \sum _{\tau \vdash [n]} (-1) ^\tau \ell(\tau)! h_\tau$$by \cite[Theorem 3.4]{RS}, which gives the formula for $e_\pi$ in the $h$-basis of $\NCSym$. However if $s_\pi$ contains $h_{1/2/\cdots / n}$ as a summand, then by the proof of Theorem~\ref{the:ncschur} this implies that $\pi = 1/2/\cdots / n$ and hence by Equation~\eqref{eq:NCSchur} it is straightforward to compute
\begin{align*} s_\pi = s_{1/2/\cdots / n} = \sum _{\svwrev{\alpha _1 + \alpha _2 + \cdots + \alpha _{\ell(\alpha)}= n}} (-1)^{n+\ell(\alpha)} \frac{1}{\alpha _1! \alpha _2! \cdots \alpha _{\ell(\alpha)}!}  h _{[\alpha _1]\slashp [\alpha _2]\slashp \cdots \slashp [\alpha _{\ell(\alpha)}]},  
\end{align*}which is a contradiction since $n\geq 3$, and hence $s_{[n]} \not\in \{ \st _\pi \} _{\pi \vdash [n], \svwrev{n\geq 0}}$ and we are done.
\end{proof}


This basis is not the only basis naturally arising from the Schur basis of $\NCSym$.
Using the action of permutations on the {Schur basis of $\NCSym$}, we can produce $n!$ different bases for $\NCSym^n$ when $n\geq 5$. 

\begin{corollary}\label{cor:permbases}
Let  $\delta$ and $\eta$ be permutations of $n\geq 5$.
\begin{enumerate}[1)]
\item The set $\{\delta\circ s_\pi: \pi\vdash [n] \}$ is a basis for $\NCSym^n$. 
\item If $\delta \neq \id$, then $\{s_\pi: \pi\vdash [n] \}$ and $\{\delta\circ s_\pi: \pi\vdash [n] \}$ are different bases for $\NCSym^n$.
\item If $\{\delta\circ s_\pi: \pi\vdash [n] \}=\{\eta\circ s_\pi: \pi\vdash [n] \},$ then $\delta=\eta$. 
\item We have that $\rho (\delta \circ s_\pi) = s_{\lambda (\pi)}$.
\end{enumerate} 
\end{corollary} 

\begin{proof} 
\begin{enumerate}[1)]
\item \svwrev{We know by Theorem~\ref{the:ncschur} that $\{ s_\pi: \pi\vdash [n] \}$ is a basis for $\NCSym^n$, so each $h_\sigma$, where $\sigma \vdash [n]$, can be written as a linear combination of the elements of $\{ s_\pi: \pi\vdash [n] \}$. Hence, $\delta\circ h_{\sigma} = h_{\delta (\sigma)}$ can be written as a linear combination of the elements of $\{\delta\circ s_\pi: \pi\vdash [n] \}$. Because $\sigma$ runs over all set partitions and $\delta$ is a bijection, it follows that $\delta(\sigma)$ runs over all set partitions. Hence the $\delta\circ s_\pi$ span the basis $\{h_\pi: \pi\vdash [n] \}$, and have the same cardinality as it, and therefore $\{\delta\circ s_\pi: \pi\vdash [n] \}$ is a basis for $\NCSym^n$.}

\item Since $\delta \neq \id$, there are integers $i$ and $i+1$ such that $\delta(i)>\delta(i+1)$.
Take the set partition $\pi=12\cdots (i-1)(i+2) \cdots n / i(i+1)$.  Then 
$$\delta_\pi= 12\cdots (i-1)(i+2) \cdots n i(i+1)$$
 and 
 \begin{align*} 
 s_{\pi}&=s_{(\delta_\pi,(n-2,2))}\\
 &=\delta_\pi\circ \left(\svwrev{\frac{1}{2(n-2)!}}h_{12\cdots (n-2) / (n-1)n}- \svwrev{\frac{1}{(n-1)!}}h_{12 \cdots (n-1) /n}\right)\\
 &=\svwrev{\frac{1}{2(n-2)!}}h_{12\cdots (i-1)(i+2) \cdots n / i(i+1)}-\svwrev{\frac{1}{(n-1)!}}h_{12\cdots i (i+2) \cdots n /(i+1)}
 \end{align*}where $n-2\geq 3$ since $n\geq5$. Note \svw{that}
  \begin{align*} 
  \delta\circ s_{\pi}&=\delta\circ s_{(\delta_\pi,(n-2,2))}  \\
   &=\delta\circ\left(\svwrev{\frac{1}{2(n-2)!}}h_{12\cdots (i-1)(i+2) \cdots n / i(i+1)}-\svwrev{\frac{1}{(n-1)!}}h_{12\cdots i (i+2) \cdots n /(i+1)}\right)\\
   & = \svwrev{\frac{1}{2(n-2)!}}h_{ \delta(1) \delta(2) \cdots \delta({i-1})\delta({i+2})\cdots \delta(n)/\delta(i)\delta({i+1}) }- \svwrev{\frac{1}{(n-1)!}}h_{\delta(1)\delta(2)\cdots \delta(i) \delta({i+2}) \cdots \delta(n) /\delta({i+1})}.
 \end{align*} 
 Now take $\sigma=\delta(1) \delta(2) \cdots \delta({i-1})\delta(i+2)\cdots \delta(n)/\delta(i)\delta({i+1}) $. Then 
 $$\delta_{\sigma}=\sort(\delta(1) \delta(2) \cdots \delta({i-1})\delta({i+2})\cdots \delta(n))\delta({i+1})\delta({i})$$ where $\sort$ is \svwrevtwo{the} function that sorts each set of positive integers increasingly. 
 We have
 \begin{align*}
 s_{\sigma}&=s_{(\delta_{\sigma},(n-2,2))}\\
 & =\delta_{\sigma}\circ \left(\svwrev{\frac{1}{2(n-2)!}}h_{12\cdots (n-2) / (n-1)n}- \svwrev{\frac{1}{(n-1)!}}h_{12 \cdots (n-1) /n}\right)\\
 &=\svwrev{\frac{1}{2(n-2)!}}h_{ \delta(1) \delta(2) \cdots \delta({i-1})\delta({i+2})\cdots \delta(n)/\delta(i)\delta({i+1}) }- \svwrev{\frac{1}{(n-1)!}}h_{\delta(1)\delta(2)\cdots \delta({i-1}) \delta({i+1}) \cdots \delta(n) /\delta({i})}.
 \end{align*}
 We  see that the leading terms for $\delta\circ s_{\pi}$
 and $s_{\sigma}$ are the same but their second terms are different. This guarantees that $\delta\circ s_{\pi}$ is not in the basis $\{s_\pi: \pi\vdash [n] \}$, but it is in  $\{\delta\circ s_\pi: \pi\vdash [n] \}$. Therefore, $\{s_\pi: \pi\vdash [n] \}$ and $\{\delta\circ s_\pi: \pi\vdash [n] \}$ are different bases.

\item Assume that 
$$\{\delta\circ s_{\pi}: \pi \vdash [n]\}=\{\eta\circ s_{\pi}: \pi \vdash [n]\}.$$ Then every $\delta \circ s_{\pi}$ is equal to some  $\eta\circ s_{\sigma}$. This implies that 
$$s_{\pi}=\delta^{-1}\svwrevtwo{\eta \circ s_{\sigma}.}$$  Therefore, \svwrevtwo{the} two bases $\{s_{\pi}: \pi \vdash [n]\}$ and $\{\delta^{-1}\eta\circ s_{\sigma}: \sigma \vdash [n]\}$ are the same. But, by the previous part, if the basis $\{\delta^{-1}\eta\circ s_{\sigma}: \sigma \vdash [n]\}$  is equal to \svwrevtwo{the basis} $\{s_{\pi}: \pi \vdash [n]\}$, then $\delta^{-1}\eta = \id$. Therefore, $\delta=\eta$.

\item By definition and Lemma~\ref{lem:rhoschur} we have that
$$\rho (\delta \circ s_\pi) = \rho(\delta \circ \delta _\pi \circ s_{[\lambda(\pi)]})= \rho(\delta\delta_\pi \circ s_{[\lambda(\pi)]}) = \rho (s_{(\delta\delta_\pi, \lambda(\pi))}) = s_{\lambda (\pi)}.$$
\end{enumerate} 
\end{proof}

We now show that our Schur functions in noncommuting variables also satisfy a product rule as our source Schur functions did. However, first we need the following. For permutations $\delta\in \Sn_n$ and $\eta\in \Sn_m$,
define the \emph{shifted concatenation} of $\delta$ and $\eta$ to be the permutation 
$$\delta \slashp \eta(i)=\begin{cases}
\delta(i) & \text{if~} 1\leq i \leq n,\\
\eta(i-n)  +n &  \text{if~} n+1\leq i \leq n+m.
\end{cases}$$ 
For example,
$$ 13425\slashp 123=13425678.$$ 

\begin{theorem}\label{the:deltaact}
Let $\delta\in \Sn_n$ and $\eta\in \Sn_m$, and let $f\in \NCSym^n$ and $g\in \NCSym^m$. Then 
$$(\delta\circ  f ) (\eta \circ g)=(\delta \slashp \eta)\circ (fg).$$
\end{theorem}

\begin{proof}
\svwrev{We will prove this in $\mathbb{Q}\langle\langle x_1, x_2, \ldots \rangle\rangle$ and hence the result will follow for $\NCSym$. We know that $\circ$ and multiplication are linear and hence if we can prove this result for two monomials $f,g$ we are done. However, this follows immediately by the definition of group action and that
$$\svwrevtwo{(\delta \slashp \eta)^{-1}}(i)=\begin{cases}
\delta^{-1}(i) & \text{if~} 1\leq i \leq n,\\
{\eta^{-1}(i-n) +n } &  \text{if~} n+1\leq i \leq n+m.
\end{cases}$$}
\end{proof}

We can now prove our product rule. 

\begin{proposition}\label{prop:prod}
Let $\delta\in \Sn_n$ and $\eta\in \Sn_m$, and let $\lambda\vdash n$ and $\mu\vdash m$. Then 
$$s_{(\delta,\lambda)}s_{(\eta,\mu)}=s_{(\delta \slashp \eta,\lambda\cdot\mu)}+s_{(\delta \slashp \eta,\lambda\odot\mu)}.$$
\end{proposition}

\begin{proof}We have by Theorems~\ref{the:prod} and \ref{the:deltaact} that 
\begin{align*} 
s_{(\delta,\lambda)}s_{(\eta,\mu)}&=(\delta \circ s_{[\lambda]})(\eta\circ s_{[\mu]})=(\delta \slashp \eta)\circ (s_{[\lambda]}s_{[\mu]})
\\
&=(\delta \slashp \eta)\circ (s_{[\lambda\cdot\mu]}+s_{[\lambda\odot \mu]})=(\delta \slashp \eta)\circ s_{[\lambda\cdot\mu]}+(\delta \slashp \eta)\circ s_{[\lambda\odot \mu]}
\\
&=s_{(\delta \slashp \eta,\lambda\cdot\mu)}+s_{(\delta \slashp \eta,\lambda\odot\mu)}.
\end{align*}
\end{proof} 

This \svwrevtwo{immediately} gives us the product rule for the Schur basis of $\NCSym$. 

\begin{corollary}\label{cor:prod}
Given set partitions $\pi$ and $\sigma$ with $\lambda(\pi)=\lambda$ and $\lambda(\sigma)=\mu$, we have that 
$$ s_\pi s_{\sigma}=s_{(\svwrevtwo{{\delta_\pi}\slashp \delta_{\sigma}},\lambda\cdot\mu)}+s_{(\svwrevtwo{{\delta_\pi}\slashp \delta_{\sigma}}, \lambda\odot \mu)}.$$
\end{corollary}

\subsection{The tabloid basis and Specht modules}\label{subsec:tabloid}  
Recall from Subsection~\ref{subsec:standard} that each pair $(\delta, \lambda/\mu)$ corresponds to a Young tableau $t$ satisfying $\sh (t) = \lambda/\mu$ and $\delta _t = \delta$ \svwrev{where $\delta \in \Sn _n$ and $\lambda/\mu$ is of size $n$, for some $n$.} Now for each Young tableau $t$ define
\begin{equation}\label{eq:schurt}
s_t= \delta _t \circ s_{[\sh(t)]}.
\end{equation}

\begin{example}\label{ex:schurt}
If $t$ is the first Young tableau in Example~\ref{ex:tabs}, then $s_t = 387219654 \circ s_{[5332/22]}$.
\end{example}

\begin{definition}\label{def:schurt} Two Young tableaux $t_1$ and $t_2$ are \emph{row equivalent} if $\sh(t_1) =\sh(t_2)$, and row $i$  of $t_1$  and row $i$ of $t_2$ contain the same elements, for all $i$. \svwrev{ This gives an equivalence relation on Young tableaux, and we denote  the equivalence class containing the Young tableau $t$ by $[t]$. If $t$ is a Young tableau, then we define}
$$s_{[t]}=\svw{\sum _{\svwrevtwo{\tilde{t} \in [t]}} s_{\tilde{t}}}$$\svwrevtwo{namely,} the sum runs over all Young tableaux $\svw{\tilde{t}}$ row equivalent to $t$.
\end{definition}

\begin{example}\label{ex:schurt2} Consider $t=\tableau{1&2\\3}$. Then $[t] = \left\{ \tableau{1&2\\3}, \tableau{2&1\\3}\right\}$ and
$$s_{\left[ \tableau{1&2\\3}\right]} = 123 \circ s_{[21]} + 213 \circ s_{[21]} = s_{\left[ \tableau{2&1\\3}\right]}.$$
\end{example}

We now reveal another basis for $\NCSym$, which we call the \emph{tabloid basis} of $\NCSym$.

\begin{theorem}\label{the:schurt}
The set
$$\{s_{[t]} : \sh(t) = \lambda (\pi), \delta _t = \delta _\pi\} _{\pi \vdash [n], \svwrev{n\geq0}}$$is  a basis for $\NCSym$. Moreover,
$$\rho(s_{[t]})= \sh (t)! s_{\sh(t)}.$$
\end{theorem}

\begin{proof} For the first part, since all $s_t$ with $\delta _t = \delta _\pi$ have the same leading term as $s_\pi$ when written in terms of the $h$-basis of $\NCSym$, we can conclude that
$$\{s_{[t]} : \sh(t) = \lambda (\pi), \delta _t = \delta _\pi\} _{\pi \vdash [n], \svwrev{n\geq 0}}$$is  a basis for $\NCSym$ as $\{ s_\pi\} _{\pi \vdash [n], \svwrev{n\geq 0}}$ is a basis for $\NCSym$ by Theorem~\ref{the:ncschur}. The second part follows immediately by Lemma~\ref{lem:rhoschur} when $\delta = \delta _t$, $\lambda = \sh(t)$, $\mu = \emptyset$ for a given $t$ and then noting that the number of Young tableaux row equivalent to $t$ is $\sh(t)!$.
\end{proof}

Note that since $\omega$ is an isomorphism
$$\{\omega(s_{[t]}) : \sh(t) = \lambda (\pi), \delta _t = \delta _\pi\} _{\pi \vdash [n], \svwrev{n\geq0}}$$is also a basis of $\NCSym$.
 
To conclude this section, we now turn our attention to Specht modules. Consider the permutation module corresponding to $\lambda$, $$M^{\lambda}=\mathbb{C}\text{-span}\{[t]: \sh(t)=\lambda\}.$$ Given a tableau $t$ with $k$ columns $C_1,C_2,\ldots,C_k$, let $$C_t=\Sn_{C_1}\times \Sn_{C_2}\times \cdots \times \Sn_{C_k}$$ be the {\it column-stabilizer} of $t$. Define 
 $$e_{t}=\left(\sum_{\delta\in C_t}{\rm sgn}(\delta) \delta\right)\circ [t]$$
and $${\bf e}_t=\left(\sum_{\delta\in C_t}{\rm sgn}(\delta) \delta\right)\circ s_{[t]}.$$ 
The corresponding  {\it Specht module} \FA{\cite[Definition 2.3.4]{sagan}} of an integer partition $\lambda$, $S^\lambda$, is the submodule of 
 $M^{\lambda}$ spanned by $$\{e_t: \sh(t)=\lambda\}.$$
 Let ${\bf S}^\lambda$ be the submodule of $\NCSym$ spanned by 
 $$\{{\bf e}_t: \sh(t)=\lambda\}.$$
 
 Now consider the map 
 $$
 \begin{array}{ccc}
 S^{\lambda}& \rightarrow & {\bf S}^\lambda \\
 e_t & \mapsto & {\bf e}_{t}.
 \end{array} 
 $$
Note that this map is surjective between two $\Sn_n$-modules, \svwrev{and by the definition of $s_t$ in Equation~\eqref{eq:schurt} we have that $\delta \circ s_t = s_{\delta t}$, where $\delta$ acts on the elements of $t$ in the natural way. Thus the map preserves the action of $\Sn _n$.} Since $S^\lambda$ is irreducible \FA{\cite[Theorem 2.4.6]{sagan}}  we conclude, \FA{if it is not zero,} that the above map is an isomorphism. 

\section{Rosas-Sagan Schur functions}\label{sec:RS}  
We now connect our (skew) Schur functions in noncommuting variables to those of Rosas-Sagan, which we state now in an equivalent form to that in \svwrevtwo{\cite[Section 6]{RS}} after we recall some well-known combinatorial concepts.

Let $\lambda/\mu$ be a skew diagram and let ${\rm SSYT}(\lambda/\mu)$ be the set of \emph{semistandard Young tableaux} $T$ of shape $\sh(T)=\lambda/\mu$; that is, the boxes of $\lambda/\mu$ are filled with positive integers such that the rows of $T$ are weakly increasing from left to right, and columns of $T$ are strictly increasing from top to bottom. Let $\nu$ be an integer partition. Then the \emph{Kostka number} $K_\nu^{\lambda/\mu}$ is the number of $T\in {\rm SSYT}(\lambda/\mu)$ such that the number of \svwrevtwo{$i$'s} appearing in $T$ is $\nu_i$. When $\mu=\emptyset$, we write ${\rm SSYT}(\lambda)$ and $K_{\nu}^\lambda$ for brevity.

\begin{definition}\label{def:RS} Let $\lambda/\mu$ be a skew diagram of size $n$. Then the \emph{Rosas-Sagan skew Schur function} $S_{\lambda/\mu}$ is
defined to be $$S_{\lambda/\mu}=\sum_{\delta\in \Sn_n}\sum_{T\in {\rm SSYT}(\lambda/\mu)} x_{c(T_{\delta(1)})}x_{c(T_{\delta(2)})}\cdots x_{c(T_{\delta(n)})}$$ where the boxes of $T$ are labelled $T_1,T_2,\ldots, T_n$, and $c(T_i)$ is the content of the box $T_i$. We often denote
$x_{c(T_{\delta(1)})}x_{c(T_{\delta(2)})}\cdots x_{c(T_{\delta(n)})}$  by $x^{(\delta,T)}$.\end{definition}

\begin{example}\label{ex:skewRS}
From the semistandard Young tableaux 
$$\tableau{1&2\\3}\qquad \svwrev{\tableau{&2\\2&3}}$$and boxes labelled naturally by row we get \FA{the monomials $x_1x_2x_3$ and \svwrev{$x_2x_2x_3$,} respectively. Then applying the permutations of $\Sn _3$ to each we get the following, respectively.}
\begin{align*}
S_{21}   &=x_1x_2x_3+x_1x_3x_2+x_2x_1x_3+x_2x_3x_1+x_3x_1x_2+x_3x_2x_1+\cdots\\
S_{22/1}&=\svwrev{x_2x_2x_3+x_2x_3x_2+x_2x_2x_3+x_2x_3x_2+x_3x_2x_2+x_3x_2x_2+\cdots}
\end{align*}
\end{example}

This is the natural generalization of their original definition, which was when $\mu = \emptyset$. This can be seen by the next lemma that is the natural generalization of \cite[Theorem 6.2 (i)]{RS}, which expresses the Rosas-Sagan Schur functions in the $m$-basis of $\NCSym$. 

\begin{lemma}\label{lem:Schur} Let $\lambda/\mu$ be a skew diagram \SXL{of size $n$}. Then 
$$S_{\lambda/\mu}=\sum_{\SXL{\nu \vdash n}} \nu! K_{\nu}^{\lambda/\mu}\sum_{\SXL{\pi \vdash [n]} \atop \lambda(\pi)=\nu}m_{\pi}.$$
\end{lemma}

\begin{proof}
Fix a $\svw{\pi=\pi_1/\pi_2/\cdots /\pi_{\ell(\pi)}  \vdash [n]}$ with $|\pi_1|\geq |\pi_2|\geq \cdots\geq |\pi_{\ell(\pi)}|$. Let $x^\pi=x_{i_1}x_{i_2}\cdots x_{i_n}$ where $i_j=p$ if and only if $j\in \pi_p$. Let $$\displaystyle S_{\lambda/\mu}=\sum_\pi c_\pi^{\lambda/\mu} m_\pi.$$ Then $c_\pi^{\lambda/\mu}$ is the coefficient of $x^\pi$ in the polynomial expansion of $S_{\lambda/\mu}$.

In order for \svwrevtwo{$x_{c(T_{\delta(1)})}x_{c(T_{\delta(2)})}\svw{\cdots }x_{c(T_{\delta(n)})}=x^\pi$,} we need semistandard Young tableaux $T$ such that the $i$th part of $\lambda(\pi)$ is the number of \svwrevtwo{$i$'s} in $T$, and permutations that fix the sets \svwrevtwo{$\{j:c(T_j)=i\}$ for all $i$.} The number of such permutations is $\svw{\lambda(\pi)!}$. Therefore,
$$S_{\lambda/\mu}=\sum_{\pi}\svw{\lambda(\pi)!}K_{\lambda(\pi)}^{\lambda/\mu}m_\pi=\sum_\nu \nu!K_\nu^{\lambda/\mu}\sum_{\lambda(\pi)=\nu}m_\pi.$$
\end{proof} 

Next, we slightly alter the way of writing our skew Schur functions \svw{in noncommuting variables} to enable our proofs to be more straightforward in this section.

A weak composition of $n$ is a \svw{list or equivalently a} tuple of nonnegative integers whose sum is $n$. \svw{The tuple notation will be useful in this section, and the list notation in the next.} For any weak composition $\alpha = \svw{\alpha _1 \alpha _2 \cdots \alpha_{\ell(\alpha)} =  (\alpha_1,\alpha_2,\ldots,\alpha_{\ell(\alpha)})}$ of  $n$, let 
\svw{$$\alpha!=\alpha_1!\alpha_2!\cdots \alpha_{\ell(\alpha)}!$$and}
$$[\alpha]=[\alpha_1]\slashp [\alpha_2]\slashp \cdots \slashp [\alpha_{\ell(\alpha)}].$$ Consider that for some $\delta,\eta\in \Sn_n$, we might have $$\delta([\alpha])=\eta([\alpha]).$$

Let $\lambda/\mu$ be a skew diagram of size $n$ and  $\varepsilon\in \Sn_{\ell(\lambda)}$. Define 
$\lambda-\mu_\varepsilon+\varepsilon-\id$ to be a tuple of length $\ell(\lambda)$ such that 
its $i$th component is $$\lambda_{i}-\mu_{\varepsilon(i)}+\varepsilon(i)-i$$ where \svwrevtwo{$\mu_{j}=0$ if $j>\ell(\mu)$.} We let $h_{[\lambda-\mu_{\varepsilon}+\varepsilon-\id]}=0$ if a component of $\lambda-\mu_{\varepsilon}+\varepsilon-\id$ is negative.
For a skew diagram $\lambda/\mu$ of size $n$ and $\delta\in \Sn_n$, observe that our skew Schur function in noncommuting variables can be written as
\begin{align}\label{eq:newskew}s_{(\delta,\lambda/\mu)}&= \sum_{\varepsilon\in \Sn_{\ell(\lambda)}} {\rm sgn}(\varepsilon)\frac{1}{\svwrev{(\lambda-\mu_{\varepsilon}-\id+\varepsilon)!}}\delta\circ h_{\svwrev{[\lambda-\mu_{\varepsilon}-\id+\varepsilon]}}\\&= \sum_{\varepsilon\in \Sn_{\ell(\lambda)}} {\rm sgn}(\varepsilon)\frac{1}{(\lambda-\mu_{\varepsilon}+\varepsilon-\id)!}\delta\circ h_{[\lambda-\mu_{\varepsilon}+\varepsilon-\id]}.\nonumber\end{align}

We will now spend most of the remainder of this section showing that skew Schur functions in noncommuting variables naturally refine Rosas-Sagan skew Schur functions, namely, $$\sum_{\delta\in \Sn_n} s_{(\delta, \lambda/\mu)}=S_{\lambda/\mu}.$$ We use a method  that is similar to the one that has been discovered and used by   \FA{Lindstr\"om \cite{Lin 73}} and Gessel and Viennot \cite{Ges, GV 85, GV pre}, \svwtwo{thus giving a completely combinatorial proof of our result that can also be proved algebraically.}

\subsection{The \FA{Lindstr\"om-Gessel-Viennot} swap}\label{subsec:LGV}

Consider the plane $\mathbb{Z}\times (\mathbb{Z}\cup \{\infty\})$. In this plane we look at the paths 
$$P=s_1s_2s_3 \cdots  $$ from $(a,1)$ to  $(a+k,b)$ where $k$ is a nonnegative integer, $b\in \mathbb{Z}\cup \{\infty\}$, and each $s_i$ is a unit length northward step ($N$) or eastward step ($E$). If a path $P$ \svwrevtwo{goes} from $(a,1)$ to  $(a+k,b)$ we write 
$$(a,1)\overset{P}{\longleftrightarrow}{(a+k,b)}.$$
 We mostly work with the paths from $(a,1)$ to  $(a+k,\infty)$. 

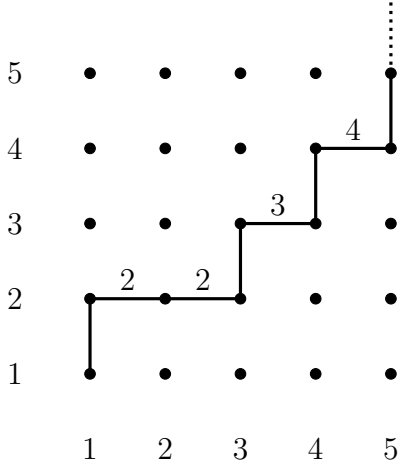
$$
\begin{tikzpicture}
    \foreach \i in {1,...,5}
     \path[black] (\i,0) node{\i} (0,\i) node{\i};
    \foreach \i in {1,...,5}
      \foreach \j in {1,...,5}{
        \draw[fill=black] (\i,\j) circle(2pt);
      };
      \draw[very thick] (1,1)--(1,2)--(2,2)--(3,2)--(3,3)--(4,3)--(4,4)--(5,4)--(5,5);
      \draw[very thick,dotted]  (5,5)--(5,6);
      
      \node at (1.5,2.25){2};
      \node at (2.5,2.25){2};
      \node at (3.5,3.25){3};
      \node at (4.5,4.25){4};
  \end{tikzpicture}
$$
The height of an eastward step $E$, denoted ${\hig}(E)$, is $b$ if $E$ \svwrevtwo{goes} from $(i,b)$ to $(i+1,b)$. The above path is $P=NEENENENN\cdots$ from $(1,1)$ to $(5,\infty)$ and the height of each eastward step is written above the step. 

\

Now for every skew diagram $\lambda/\mu$  and \FA{$\varepsilon\in \Sn_{\ell(\lambda)}$,} let  $\mathcal{P}(\varepsilon,\lambda/\mu)$ be the set of all tuples of paths $$(P_1,P_2,\ldots, P_{\ell(\lambda)})$$ such that $P_i$ is a path from $(\mu_{\varepsilon(i)}-\varepsilon(i),1)$ to $(\lambda_i-i,\infty)$, 
$$(\mu_{\varepsilon(i)}-\varepsilon(i),1)\overset{P_i}{\longleftrightarrow} (\lambda_i-i,\infty).$$
For example, if $\lambda/\mu=332/110$ and $\varepsilon=213$, one of the tuples in $\mathcal{P}(213,332/110)$ is \svwrevtwo{$P=(P_1,P_2,P_3)$} where
$$P_1=NEENENN\cdots, \quad\quad\quad  P_2=NNENNNN\cdots,   \quad\quad\quad P_3=ENNENNN\cdots$$
such that
$$(-1,1)\overset{P_1}{\longleftrightarrow} (2,\infty),~~~~~~\quad\quad \quad ~~~~~(0,1)\overset{P_2}{\longleftrightarrow} (1,\infty),~~~~~~\quad\quad \quad~~~~~~(-3,1)\overset{P_3}{\longleftrightarrow} (-1,\infty).
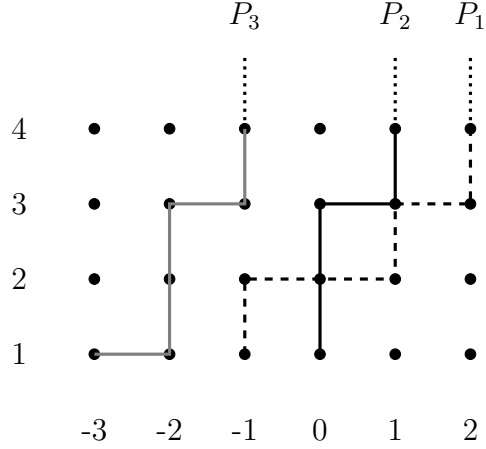
$$

$$
\begin{tikzpicture}
    \foreach \i in {-3,...,2}
     \path[black] (\i,0) node{\i};
     \foreach \i in {1,...,4}
     \path[black] (-4,\i) node{\i};
    \foreach \i in {-3,...,2}
      \foreach \j in {1,...,4}{
        \draw[fill=black] (\i,\j) circle(2pt);
      };
      
      \draw[very thick,dashed] (-1,1)--(-1,2)--(0,2)--(1,2)--(1,3)--(2,3)--(2,4);
      \draw[very thick,dotted,black]  (2,4)--(2,5);
      \node at (2,5.5){$P_1$};
      
        \draw[very thick,black] (0,1)--(0,2)--(0,3)--(1,3)--(1,4);
      \draw[very thick,dotted,black]  (1,4)--(1,5);
       \node at (1,5.5){$P_2$};
      
            \draw[very thick,gray] (-3,1)--(-2,1)--(-2,2)--(-2,3)--(-1,3)--(-1,4);
      \draw[very thick,dotted,black]  (-1,4)--(-1,5);
         \node at (-1,5.5){$P_3$};
  \end{tikzpicture}
$$
\svwrevtwo{Let} $$\mathcal{P}(\lambda/\mu)=\bigcup_{\varepsilon\in \Sn_{\ell(\lambda)}}\mathcal{P}(\varepsilon,\lambda/\mu).$$ Consider that if $P\in \mathcal{P}(\lambda/\mu)$, then by looking at the initial and end points of $P$, we can  deduce for which $\varepsilon\in \Sn_{\ell(\lambda)}$, $P$ is in $\mathcal{P}(\varepsilon,\lambda/\mu)$. We define $${\rm sgn}\svw{(P)}={{\rm sgn}(\varepsilon)}.$$ 

\

We now define the \FA{Lindstr\"om-Gessel-Viennot} swap, which is an involution on $\mathcal{P}(\lambda/\mu)$.

\begin{definition}\label{def:LGVswap}{\bf (\FA{Lindstr\"om-Gessel-Viennot} swap)}
Given $P=(P_1,P_2,\ldots,P_{\ell(\lambda)})\in \mathcal{P}(\varepsilon,\lambda/\mu)$, define $\iota(P)=P'=(P'_1,P'_2,\ldots,P'_{\ell(\lambda)})$ as follows.
\begin{enumerate}[1)]
\item If $P_i$ and  $P_j$ do not intersect for all $i,j$, then $P'=P$.
\item \svwtwo{Otherwise, find the smallest index $i$ such that the path $P_i$ intersects some other paths. Starting from $(\mu_{\varepsilon(i)}-\varepsilon(i),1)$ and going along $P_i$, let $(a,b)$ be the {last} place where $P_i$ intersects with another path. Finally, find the smallest index $j>i$ such that $P_i$ intersects $P_j$ at $(a,b)$. Now let} 
\begin{align*}
P'_i&=\text{~the initial point of ~}P_j\overset{P_j}{\longleftrightarrow}(a,b)\overset{P_i}{\longleftrightarrow} \text{~the end point of~} P_i\\
&=(\mu_{\varepsilon(j)}-\varepsilon(j),1)\overset{P_j}{\longleftrightarrow}(a,b)\overset{P_i}{\longleftrightarrow}(\lambda_i-i,\infty),
\end{align*}
and
\begin{align*}
P'_j&=\text{~the initial point of ~}P_i\overset{P_i}{\longleftrightarrow}(a,b)\overset{P_j}{\longleftrightarrow} \text{~the end point of~} P_j\\
&=(\mu_{\varepsilon(i)}-\varepsilon(i),1)\overset{P_i}{\longleftrightarrow}(a,b)\overset{P_j}{\longleftrightarrow}(\lambda_j-j,\infty). 
\end{align*}
Now define, $P'=P$ with $P_i$ and $P_j$ replaced by $P'_i$ and $P'_j$, respectively.
\end{enumerate}
\end{definition}

$$
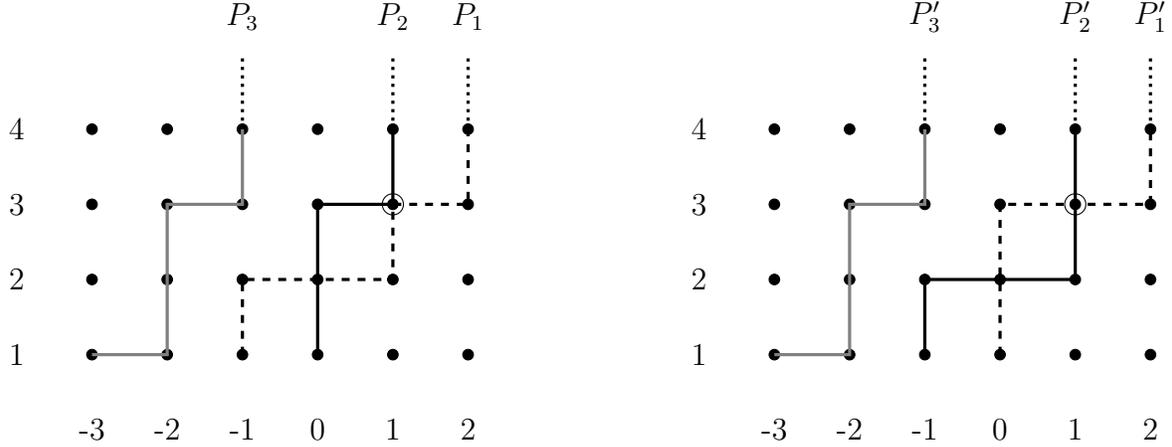
\begin{tikzpicture}
    \foreach \i in {-3,...,2}
     \path[black] (\i,0) node{\i};
     \foreach \i in {1,...,4}
     \path[black] (-4,\i) node{\i};
    \foreach \i in {-3,...,2}
      \foreach \j in {1,...,4}{
        \draw[fill=black] (\i,\j) circle(2pt);
      };
      \draw (1,3) circle(4pt);
      \draw[very thick,dashed] (-1,1)--(-1,2)--(0,2)--(1,2)--(1,3)--(2,3)--(2,4);
      \draw[very thick,dotted]  (2,4)--(2,5);
      \node at (2,5.5){$P_1$};
      
        \draw[very thick,black] (0,1)--(0,2)--(0,3)--(1,3)--(1,4);
      \draw[very thick,dotted,black]  (1,4)--(1,5);
       \node at (1,5.5){$P_2$};
      
            \draw[very thick,gray] (-3,1)--(-2,1)--(-2,2)--(-2,3)--(-1,3)--(-1,4);
      \draw[very thick,dotted,black]  (-1,4)--(-1,5);
         \node at (-1,5.5){$P_3$};
  \end{tikzpicture}
  \quad \quad \quad \quad \quad \quad 
  \begin{tikzpicture}
    \foreach \i in {-3,...,2}
     \path[black] (\i,0) node{\i};
     \foreach \i in {1,...,4}
     \path[black] (-4,\i) node{\i};
    \foreach \i in {-3,...,2}
      \foreach \j in {1,...,4}{
        \draw[fill=black] (\i,\j) circle(2pt);
      };
      \draw (1,3) circle(4pt);
      \draw[very thick,black] (-1,1)--(-1,2)--(1,2)--(1,2)--(1,4);
      \draw[very thick,dotted,black]  (1,4)--(1,5);
      \node at (2,5.5){$P'_1$};
      
        \draw[very thick,dashed] (0,1)--(0,3)--(2,3)--(2,4);
      \draw[very thick,dotted,black]  (2,4)--(2,5);
       \node at (1,5.5){$P'_2$};
      
            \draw[very thick,gray] (-3,1)--(-2,1)--(-2,2)--(-2,3)--(-1,3)--(-1,4);
      \draw[very thick,dotted,black]  (-1,4)--(-1,5);
         \node at (-1,5.5){$P'_3$};
  \end{tikzpicture} 
$$

The \FA{Lindstr\"om-Gessel-Viennot} swap $\iota$ guarantees that 
\begin{enumerate}[1)]
\item \svwtwo{$\iota$ is an involution, that is, $\iota^2(P)=P$ for all $P$,} 
\item if $P'=\iota(P)=P$, then there are no mutually distinct paths in $P$ that intersect, and furthermore $P\in \mathcal{P}(\id,\lambda/\mu)$,
\item  if $P'=\iota(P)\neq P$, then  ${\rm sgn}(P)=-{\rm sgn}(P').$
\end{enumerate} 

\begin{definition}\label{def:Plabel} 
Let $\lambda/\mu$ be a skew diagram of size $n$. Given a tuple $P=(P_1,P_2,\ldots,P_{\ell(\lambda)})\in \mathcal{P}(\varepsilon,\lambda/\mu),$ label the eastward steps by ${ 1}_P,{2}_P,{ 3}_P,\ldots, n_P$ such that
\begin{enumerate}[1)]
\item the labels of \svwrevtwo{the} eastward steps in $P_i$ are smaller than the labels of \svwrevtwo{the} eastward steps in $P_j$ if $i<j$,
\item in each $P_i$, label the eastwards steps increasingly from left to right, bottom to top.
\end{enumerate} 
We identify each eastward step by its label. Moreover, for any path $P$, 
define the label exchange permutation ${\exch}_P$ in $\Sn_n$ such that $\exch_P(i)=j$ if the eastward step labelled by $i_P$ in $P$ is labelled by $j_P'$ in $\iota(P)=P'$. 
\end{definition}
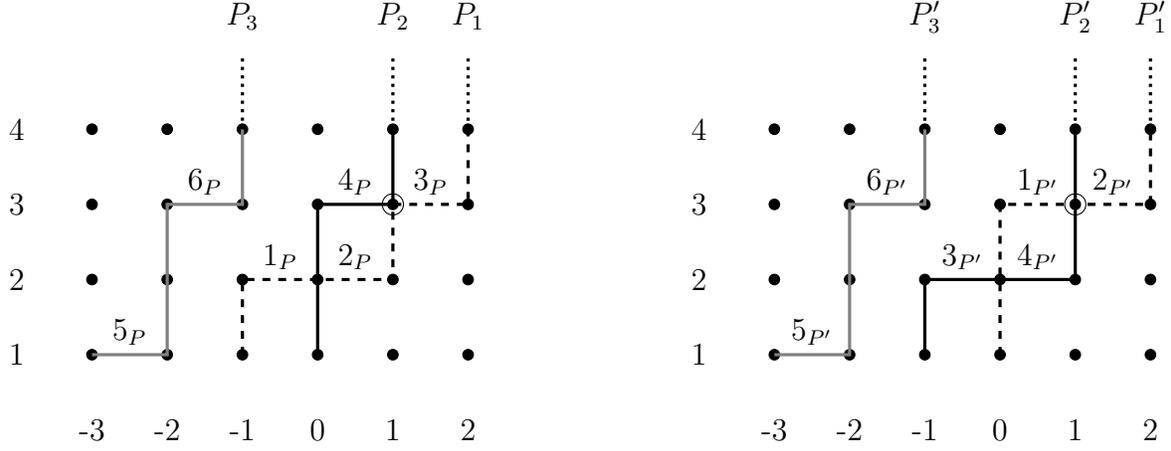
$$
\begin{tikzpicture}
    \foreach \i in {-3,...,2}
     \path[black] (\i,0) node{\i};
     \foreach \i in {1,...,4}
     \path[black] (-4,\i) node{\i};
    \foreach \i in {-3,...,2}
      \foreach \j in {1,...,4}{
        \draw[fill=black] (\i,\j) circle(2pt);
      };
      \draw (1,3) circle(4pt);
      \draw[very thick,dashed] (-1,1)--(-1,2)--(0,2)--(1,2)--(1,3)--(2,3)--(2,4);
      \draw[very thick,dotted]  (2,4)--(2,5);
      \node at (2,5.5){$P_1$};
      \node at (-0.5,2.3){${ 1_P}$};
       \node at (0.5,2.3){${2_P}$};
        \node at (1.5,3.3){${ 3_P}$};
      
        \draw[very thick,black] (0,1)--(0,2)--(0,3)--(1,3)--(1,4);
      \draw[very thick,dotted,black]  (1,4)--(1,5);
       \node at (1,5.5){$P_2$};
       \node at (0.5,3.3){${4_P}$};
      
            \draw[very thick,gray] (-3,1)--(-2,1)--(-2,2)--(-2,3)--(-1,3)--(-1,4);
      \draw[very thick,dotted,black]  (-1,4)--(-1,5);
         \node at (-1,5.5){$P_3$};
         \node at (-2.5,1.3){${5_P}$};
         \node at (-1.5,3.3){${6_P}$};
  \end{tikzpicture}
  \quad \quad \quad \quad \quad \quad 
  \begin{tikzpicture}
    \foreach \i in {-3,...,2}
     \path[black] (\i,0) node{\i};
     \foreach \i in {1,...,4}
     \path[black] (-4,\i) node{\i};
    \foreach \i in {-3,...,2}
      \foreach \j in {1,...,4}{
        \draw[fill=black] (\i,\j) circle(2pt);
      };
      \draw (1,3) circle(4pt);
      \draw[very thick,black] (-1,1)--(-1,2)--(1,2)--(1,2)--(1,4);
      \draw[very thick,dotted,black]  (1,4)--(1,5);
      \node at (2,5.5){$P'_1$};
      \node at (-0.5,2.3){${3}_{P'}$};
       \node at (0.5,2.3){${4}_{P'}$};
       \node at (1.5,3.3){${2}_{P'}$};
      
        \draw[very thick,dashed] (0,1)--(0,3)--(2,3)--(2,4);
      \draw[very thick,dotted,black]  (2,4)--(2,5);
       \node at (1,5.5){$P'_2$};
       \node at (0.5,3.3){${1}_{P'}$};
      
            \draw[very thick,gray] (-3,1)--(-2,1)--(-2,2)--(-2,3)--(-1,3)--(-1,4);
      \draw[very thick,dotted,black]  (-1,4)--(-1,5);
         \node at (-1,5.5){$P'_3$};
         \node at (-2.5,1.3){${5}_{P'}$};
          \node at (-1.5,3.3){${ 6}_{P'}$};
  \end{tikzpicture} 
$$
$$\exch_P(1)=3, \exch_P(2)=4, \exch_P(3)=2, \exch_P(4)=1, \exch_P(5)=5, \exch_P(6)=6.$$

\begin{definition}\label{def:xdP}
Let $\lambda/\mu$ be a skew diagram of  size $n$. For any $P=(P_1,P_2,\ldots,P_{\ell(\lambda)})\in \mathcal{P}(\varepsilon,\lambda/\mu),$ and $\delta\in \Sn_n$, define 
$$x^{(\delta,P)}=x_{\hig(\delta({1})_P)}x_{\hig(\delta({ 2})_P)}\cdots x_{\hig(\delta({ n})_P)}.$$\end{definition} 

As an example, if $\delta=315462$ and $P$ is the above path \svw{tuple}, then  
\begin{align*} x^{(\delta,P)}&=x_{\hig(\delta({1})_P)}x_{\hig(\delta({ 2})_P)}\cdots x_{\hig(\delta({6})_P)}\\
&=x_{\hig({ 3}_P)}x_{\hig({ 1}_P)}x_{\hig({ 5}_P)}x_{\hig({4}_P)}x_{\hig({ 6}_P)}x_{\hig({2}_P)}\\
&=x_3x_2x_1x_3x_3x_2.
\end{align*}

\subsection{The main result}\label{subsec:main} To show our main result, we first need some lemmas. 

\begin{lemma}\label{lem:hmon1} Let $\lambda/\mu$ be a skew diagram of size $n$ and \FA{$\varepsilon\in \Sn_{\ell(\lambda)}$.} Then 
$$\delta\circ h_{[\lambda-\mu_\varepsilon+\varepsilon-\id]}=\sum_{P\in \mathcal{P}(\varepsilon,\lambda/\mu)}\sum_{\svwrev{\eta([\lambda-\mu_\varepsilon+\varepsilon-\id])=\delta([\lambda-\mu_\varepsilon+\varepsilon-\id])}} x^{(\eta,P)}.$$
\end{lemma}

\begin{proof}
	The $i$th part of $\lambda-\mu_\varepsilon+\varepsilon-\id$ has size $\lambda_i-\mu_{\varepsilon(i)}+\varepsilon(i)-i$ = $(\lambda_i-i) - (\mu_{\varepsilon(i)}-\varepsilon(i)) =$ the number of eastward \svwrev{$E$} steps in $P_i$. \svwrev{ Observe that the second summation is over all $\eta\in\mathfrak{S}_n$ that fixes blocks of $\delta([\lambda-\mu_\epsilon+\epsilon-\id])$.} Comparing this with Lemma \ref{lem:h} completes the proof.
\end{proof}

\begin{lemma}\label{lem:hmon2}
Let $\lambda/\mu$ be a skew diagram of size $n$ and $\varepsilon\in \Sn_{\ell(\lambda)}$. Then
$$\sum_{\delta\in \Sn_n}\frac{1}{(\lambda-\mu_\varepsilon+\varepsilon-\id)!}\delta \circ h_{[\lambda-\mu_\varepsilon+\varepsilon-\id]}=\sum_{P\in \mathcal{P}(\varepsilon,\lambda/\mu)}\sum_{\delta\in \Sn_n} x^{(\delta,P)}.$$
\end{lemma}

\begin{proof}
	Note \svwrevtwo{the following.} $$|\{\eta\in\Sn_n:\eta([\lambda-\mu_\varepsilon+\varepsilon-\id])=\delta([\lambda-\mu_\varepsilon+\varepsilon-\id])\}|=(\lambda-\mu_\varepsilon+\varepsilon-\id)!$$ Considering Lemma \ref{lem:hmon1}, for each $\eta\in\Sn_n$ and $P\in(\varepsilon,\lambda/\mu)$, the number of $\delta\in \Sn_n$ that gives ${x}^{(\eta,P)}$ is $(\lambda-\mu_\varepsilon+\varepsilon-\id)!$. Therefore,
	\begin{align*}
	LHS = & \sum_{\delta\in\Sn_n}\sum_{P\in \mathcal{P}(\varepsilon,\lambda/\mu)}\sum_{\svwrev{\eta([\lambda-\mu_\varepsilon+\varepsilon-\id])=\delta([\lambda-\mu_\varepsilon+\varepsilon-\id])}}\frac{x^{(\eta,P)}}{(\lambda-\mu_\varepsilon+\varepsilon-\id)!}\\
	&= \sum_{\eta\in\Sn_n}\sum_{P\in \mathcal{P}(\varepsilon,\lambda/\mu)}\frac{(\lambda-\mu_\varepsilon+\varepsilon-\id)!x^{(\eta,P)}}{(\lambda-\mu_\varepsilon+\varepsilon-\id)!}=RHS.
	\end{align*}
\end{proof}

\begin{lemma}\label{lem:swapmon}
Let $\lambda/\mu$ be a skew diagram of size $n$ and $\varepsilon\in \Sn_{\ell(\lambda)}$. Let $P\in \mathcal{P}(\varepsilon,\lambda)$ and $\iota(P)=P'$. 
 For any permutation $\delta\in \Sn_n$, we have  the following. 
\begin{enumerate}[1)]
\item $\hig(i_P)=\hig(\exch_P(i)_{P'})$
\item $x^{(\delta,P)}=x^{(\exch_P\delta,P')}$
\item $$\sum_{(\delta,P)\in \Sn_n\times \mathcal{P}(\lambda/\mu)} {{\rm sgn}(P)}x^{(\delta,P)}
	=\sum_{P\in\mathcal{P}(\id,\lambda/\mu)\atop P\text{ has no self-intersection}}\sum_{\delta\in\Sn_n}{x}^{(\delta,P)}$$
\end{enumerate}
\end{lemma}
\begin{proof} 
\begin{enumerate}[1)]
\item This follows directly from the definition of $\exch_P$ since it preserve the height. 
\item Note that 
\begin{align*}
x^{(\delta,P)}&=x_{\hig(\delta({1})_P)}x_{\hig(\delta({ 2})_P)}\cdots x_{\hig(\delta({ n})_P)}\\
& =x_{\hig(\exch_P\delta({1})_{P'})}x_{\hig(\exch_P\delta({ 2})_{P'})}\cdots x_{\hig(\exch_P\delta({ n})_{P'})}& ({\rm by ~ Part\ 1})\\
&=x^{(\exch_P\delta,P')}.
\end{align*} 
\item Considering the \FA{Lindstr\"om-Gessel-Viennot} swap $\iota$, if $\iota(P)=P'\neq P$, 
then ${\rm sgn}(P)=-{\rm sgn}(P')$. Thus by \svw{the second part,}
$${\rm sgn}(P)x^{(\delta,P)}=-{\rm sgn}(P')x^{(\exch_P\delta,P')}.$$ Therefore, 
\svwrev{since $\iota$ is an involution, all terms in the sum cancel except those corresponding to nonintersecting paths.}
\end{enumerate} 
\end{proof} 

We can now  state our main result. 

\begin{theorem}\label{the:RSrefines}
Let $\lambda/\mu$ be a skew diagram of size $n$. Then 
$$\sum_{\delta\in \Sn_n}s_{(\delta,\lambda/\mu)}=S_{\lambda/\mu}.$$ 
In particular, when $\mu=\emptyset$, we have that
$$\sum_{\delta\in \Sn_n}s_{(\delta,\lambda)} = \sum_{\svw{t: \sh(t)=\lambda}}s_{t} = \sum_{\svwrevtwo{\textrm{distinct classes} \atop [t]: \sh(t)=\lambda}}s_{[t]}= S_\lambda.$$
\end{theorem}
\begin{proof}
We have that
\begin{align*}
	\sum_{\delta\in\Sn_n}\svw{s}_{(\delta,\lambda/\mu)}&=\sum_{\delta\in\Sn_n}\sum_{\varepsilon\in\Sn_{\ell(\lambda)}}{\rm sgn}(\varepsilon) \frac{1}{(\lambda-\mu_\varepsilon+\varepsilon-\id)!}\delta\circ h_{[\lambda-\mu_\varepsilon+\varepsilon-\id]}\\
	&=\sum_{\varepsilon\in\Sn_{\ell(\lambda)}}\sum_{P\in \mathcal{P}(\varepsilon,\lambda/\mu)}\sum_{\delta\in\Sn_n}{\rm sgn}(\varepsilon){x}^{(\delta,P)} &(\text{by Lemma~\ref{lem:hmon2}})\\
	&=\sum_{(\delta,P)\in \Sn_n\times \mathcal{P}(\lambda/\mu)} {{\rm sgn}(P)}x^{(\delta,P)} 
	\\
	&=\sum_{P\in\mathcal{P}(\id,\lambda/\mu)\atop P\text{ has no self-intersection}}\sum_{\delta\in\Sn_n}{x}^{(\delta,P)} &(\text{by Lemma~\ref{lem:swapmon} 3)})\\
	&=\sum_{T\in {\rm SSYT}(\lambda/\mu)}\sum_{\delta\in\Sn_n}{x}^{(\delta,T)}\\
	&=S_{\lambda/\mu}
	\end{align*}
	where the penultimate equality follows from, for example, the proof of Theorem 7.16.1 in \cite{EC2}.\end{proof}
	
\begin{corollary}\label{cor:rhoRS}
Let $\lambda /\mu$ be a skew diagram of size $n$. Then
$$\rho (S_{\lambda /\mu}) = n! s _{\lambda/\mu}.$$
\end{corollary}

\begin{proof} This follows immediately from Theorem~\ref{the:RSrefines} and Lemma~\ref{lem:rhoschur}.
\end{proof}


There are two celebrated formulae involving classical skew Schur functions. The first of these is the Littlewood-Richardson rule, which expresses a skew Schur function $s_{\lambda/\mu}$ as a sum of Schur functions \svw{$s_\nu$}
$$s_{\lambda/\mu} = \sum _\nu c_{\mu\nu}^{\lambda} s_\nu.$$ The $c_{\lambda\mu}^{\nu}$ are nonnegative integers known as Littlewood-Richardson coefficients, and the interested reader may consult \cite{EC2} for a myriad of ways to compute them combinatorially. The second of these is the coproduct formula for \svw{a Schur function $s_\lambda$ in terms of Schur functions $s_\mu$}
$$\Delta (s_\lambda) = \sum _\mu s_\mu \otimes s_{\lambda/\mu}.$$We conclude \svw{this section} by showing that analogous results hold for Rosas-Sagan skew Schur functions, after we prove the following lemma.

\begin{lemma}\label{lem:skew}Kostka numbers satisfy $$K_\gamma^{\lambda/\mu}=\displaystyle\sum_\nu c^\lambda_{\mu\nu}K_\gamma^\nu$$where \svw{$\lambda, \mu, \gamma, \nu$ are integer partitions and} $c_{\mu\nu}^\lambda$ are the Littlewood-Richardson coefficients.
\end{lemma}

\begin{proof}
	This follows from the well-known identities \cite[Equations (7.35), (7.36), (A1.142)]{EC2}  $\displaystyle s_\nu=\sum_\gamma K_\gamma^\nu m_\gamma$, $\displaystyle s_{\lambda/\mu}=\sum_\gamma K^{\lambda/\mu}_\gamma m_\gamma$ and $\displaystyle s_{\lambda/\mu}=\sum_\nu c_{\mu\nu}^\lambda s_\nu$.
\end{proof}

\begin{proposition}\label{prop:RSLR}
	Rosas-Sagan skew Schur functions can be written as a positive sum of Rosas-Sagan Schur functions, namely, \svw{if $\lambda/\mu$ is a skew diagram of size $n$, then}
	$$S_{\lambda/\mu}=\sum_{\svw{\nu \vdash n}} c^\lambda_{\mu\nu}S_\nu$$where \svw{$c_{\mu\nu}^\lambda$ are the Littlewood-Richardson coefficients.}
\end{proposition}

\begin{proof}
	It follows from Lemma \ref{lem:Schur} and Lemma \ref{lem:skew} that
	\begin{align*}
	S_{\lambda/\mu}&=\sum_{\svwrevtwo{\gamma \vdash n}} \gamma! K^{\lambda/\mu}_\gamma \sum_{\svwrevtwo{\pi \vdash [n]} \atop\lambda(\pi)=\gamma}m_\pi\\
	&=\sum_{\svwrevtwo{\gamma \vdash n}}\gamma!\sum_{\svwrevtwo{\nu \vdash n}} c_{\mu\nu}^\lambda K_\gamma^\nu \sum_{\svwrevtwo{\pi \vdash [n]} \atop\lambda(\pi)=\gamma}m_\pi\\
	&=\sum_{\svwrevtwo{\nu \vdash n}} c^\lambda_{\mu\nu}S_\nu.
	\end{align*}
\end{proof}

\svwrev{For the following we use the coproduct definition on $\NCSym$ as given in \cite[Section 2.4]{28authors}.}


\begin{proposition}\label{prop:RScoprod}\svwtwo{Let $\Delta_{k,n-k}:\NCSym^n\to\NCSym^k\otimes\NCSym^{n-k}$ be the coproduct on \svwrev{$\NCSym$ restricted to the homogeneous component $\NCSym^k\otimes\NCSym^{n-k}$.} Then for $\lambda\vdash n$, we have that}
	$$\Delta_{k, n-k}(S_\lambda) = \binom{n}{k}\sum_{\mu\vdash k \atop \svw{\mu \subseteq\lambda}} S_\mu\otimes S_{\lambda/\mu}.$$
\end{proposition}

\begin{proof}
	Let $T\in\rm{SSYT}(\lambda)$ with boxes labelled $T_1,  \ldots,T_n$. Consider $\mu\subseteq \lambda$ with $\mu\vdash k$ and the boxes in $T$ corresponding to the shapes $\mu$, $\lambda/\mu$ are $\{T_{i_1},\ldots,T_{i_k}\}$ and $\{T_{j_1},\ldots,T_{j_{n-k}}\}$ respectively. Let $\tau\in\Sn_n$. Then we denote $x_{c(T_{i_{\tau(1)}})}\cdots x_{c(T_{i_{\tau(k)}})}$ and $x_{c(T_{j_{\tau(1)}})}\cdots x_{c(T_{j_{\tau(n-k)}})}$ by $x^{(\tau|_\mu,T)}$ and $x^{(\tau|_{\lambda/\mu},T)}$ respectively.
	
	For convenience, we set $\{i_1,\ldots,i_k\}$ to be $\{1,\ldots,k\}$ and $\{j_1,\ldots,j_{n-k}\}$ to be $\{k+1,\ldots,n\}$. In this case, $\tau|_{\mu}={\rm std}(\tau(1)\cdots\tau(k))$ and $\tau|_{\lambda/\mu}={\rm std}(\tau(k+1)\cdots\tau(n))$, that is, standardized using their relative order.
	
	Consider the coproduct
	$$\displaystyle S_\lambda(\mathbf{x,y}) = \sum_{\tau\in\Sn_n}\sum_{T\in \rm{SSYT}(\lambda)}(x,y)^{(\tau,T)}=\sum_{\tau\in\Sn_n}\sum_{\mu\subseteq\lambda}\sum_{T\in \rm{SSYT}(\mu)}\sum_{S\in \rm{SSYT}(\lambda/\mu)}x^{(\tau|_\mu,T)}y^{(\tau|_{\lambda/\mu},S)}.$$
	Hence,
	$$\Delta_{k,n-k}(S_\lambda) = \sum_{\tau\in\Sn_n} \sum_{\mu\vdash k \atop \mu\subseteq\lambda}\left( \sum_{T\in \rm{SSYT}(\mu)} x^{(\tau|_\mu,T)}\otimes \sum_{T\in \rm{SSYT}(\lambda/\mu)}x^{(\tau|_{\lambda/\mu},T)}\right).$$
	For each $\tau'\in \Sn_k,\tau''\in\Sn_{n-k}$, the number of $\tau\in\Sn_n$ such that $\tau|_\mu=\tau'$ and $\tau|_{\lambda/\mu}=\tau''$ is $\displaystyle\binom{n}{k}$. Therefore, the equation above is equal to
	$$\binom{n}{k}\sum_{\tau'\in\Sn_k}\sum_{\tau''\in\Sn_{n-k}}\sum_{\mu\vdash k \atop \mu\subseteq\lambda}\left( \sum_{T\in \rm{SSYT}(\mu)} x^{(\tau',T)}\otimes \sum_{T\in \rm{SSYT}(\lambda/\mu)}x^{(\tau'',T)} \right)=\binom{n}{k}\sum_{\mu\vdash k \atop \mu\subseteq\lambda} S_\mu\otimes S_{\lambda/\mu}.$$
\end{proof}
\section{Immaculate and noncommutative ribbon Schur functions}\label{sec:NSym}
In our final section we turn our attention to another Hopf algebra of  $\bQ \langle \langle x_1, x_2, \ldots \rangle\rangle$, but before we do we need a few more combinatorial concepts. Given a positive integer $n$, we say that a \emph{composition} $\alpha = \alpha _1 \alpha _2 \cdots \alpha _{\ell(\alpha)}$ of $n$ is an ordered list of positive integers whose sum is $n$. We denote this by $\alpha \vDash n$, call the $\alpha _i$ the \emph{parts} of $\alpha$, $\ell(\alpha)$ the \emph{length} of $\alpha$, and $n$ the \emph{size} of $\alpha$. We define $\lambda(\alpha)$ to be the \svw{integer} partition obtained from $\alpha$ by writing the parts of $\alpha$ in weakly decreasing order. We define $\alpha !$ and and $[\alpha]$ as in the previous section for weak compositions. Given two compositions $\alpha=\alpha_1\alpha_2 \cdots\alpha_{\ell(\alpha)}$ and $\beta=\beta_1\beta _2 \cdots\beta_{\ell(\beta)}$, the \textit{concatenation} of $\alpha$ and $\beta$ is $\alpha\cdot\beta = \alpha_1\alpha_2\cdots\alpha_{\ell(\alpha)}\beta_1\beta _2\cdots\beta_{\ell(\beta)}$, while their \textit{near concatenation} is $\alpha\odot\beta =  \alpha_1\alpha_2\cdots(\alpha_{\ell(\alpha)}+\beta_1)\beta _2\cdots\beta_{\ell(\beta)}$. We say $\alpha$ is a \textit{coarsening} of $\beta$  (or equivalently $\beta$ is a \textit{refinement} of $\alpha$), denoted by $\alpha \succcurlyeq\beta$, if we can obtain the parts of $\alpha$ in order by adding together adjacent parts of $\beta$ in order.

\begin{example} \label{ex:comps} If $\alpha = 2312 \vDash 8$, then $\lambda (\alpha) =3221$, $\alpha ! = 2!3!1!2! = 24$, and $[\alpha ] = [2] \slashp [3] \slashp [1] \slashp [2] = 12/345/6/78$. If $\beta =12$, then $\alpha \cdot \beta = 231212$, $\alpha \odot \beta = 23132$ and $23132 \succcurlyeq 231212$.
\end{example}

We can now define the last Hopf algebra of interest to us, the graded \emph{Hopf algebra of noncommutative symmetric functions}, $\NSym$,
$$\NSym  = \NSym ^0 \oplus \NSym ^1 \oplus \cdots \subset \bQ \langle \langle x_1, x_2, \ldots \rangle\rangle$$where $\NSym ^0 = \spam \{1\}$ and {the} $n$th graded piece for $n\geq 1$ has the following original bases \cite{GKLLRT}
\begin{align*}
\svw{\NSym ^n}&= \spam\{ \nch_\alpha \suchthat \alpha \vDash n\} = \spam\{ \ncr_\alpha \suchthat \alpha\vDash n\}
\end{align*}where these functions are defined as follows, given a composition $\alpha = \alpha _1 \alpha _2 \cdots \alpha _{\ell(\alpha)} \vDash n$.

The \emph{complete homogeneous noncommutative symmetric function}, $\nch_\alpha$, is given by
\svw{$$\nch_\alpha = \nch_{\alpha _1}\nch_{\alpha _2}\cdots \nch_{\alpha _{\ell(\alpha)}}$$where \svwrevtwo{$\nch_{i} = \sum_{j_1\leq j_2\leq\cdots \leq j_{i}} x_{j_1}x_{j_2}\cdots x_{j_{i}}$.}}

\begin{example}\label{ex:nccompletes}
$\nch_{12} = (x_1+x_2+\cdots)(x_1x_2+x_1^2+ \cdots)$
\end{example}

The \emph{noncommutative ribbon Schur function}, $\ncr _\alpha$, is given by 
$$\ncr_{\alpha} = (-1)^{\ell(\alpha)} \sum _{\beta \succcurlyeq \alpha} (-1)^{\ell (\beta)}  \nch_{\beta }.$$

\begin{example}\label{ex:ncribbons}
$\ncr_{12}=\nch _{12} - \nch _3$
\end{example}

Many other bases exist in analogy to those in $\Sym$, however, it is these two that are most relevant to us, in addition to one other \svw{consisting of immaculate functions}. The \emph{immaculate function}, $\Sn _\alpha$, is given by \cite[Theorem 3.27]{BBSSZ}
$$\Sn_{\alpha}=\sum_{\varepsilon\in \Sn_{\ell(\alpha)}} {\rm sgn}(\varepsilon) \nch_{\alpha+\varepsilon-{\id}}$$ where $\alpha+\varepsilon-{\id}$ is a list of length $\ell(\alpha)$ such that its $i$th part is $$\alpha_i+\varepsilon(i)-i.$$We  let  $\nch _0 = 1$ and $\nch_{\alpha+\varepsilon-{\id}} = 0$ if any $\alpha_i+\varepsilon(i)-i$ is negative.

We now relate our three Hopf algebras. Consider the following diagram,
$$
\begin{tikzpicture}
\node(N) at (0,0){${\NSym}$};
\node(NC) at (0,2){${\NCSym}$};
\node(SU) at (3,2){${\Sym}$};
\node at (-0.2,1){$\jota$};
\node at (1.5,2.2){$\rho$};
\node at (1.6,0.8){$\chi$};

\draw[thick,->] (N)--(SU);
\draw[thick,->] (NC)--(SU);
\draw[thick,->] (N)--(NC);
\end{tikzpicture} 
$$
where, for a composition $\alpha$, the map $\jota$ is  given by 
$$\jota(\nch_\alpha)=\frac{1}{\alpha!}h_{[\alpha]},$$ and the map $\chi$ that lets the variables commute is given by $$\chi(\nch_\alpha)=h_{\lambda(\alpha)}.$$\sagan{The map $\jota$ is a Hopf morphism that was originally defined in the $m$-basis \cite[Theorem 4.6]{BRRZ}. However, for our purposes the equivalent definition in the $h$-basis is most useful.}
Note the above diagram commutes since \svwrevtwo{by Lemma~\ref{lem:RSrho} 4) we have that}
$$\rho(\jota(\nch_\alpha))=\rho\left(\frac{1}{\alpha!}h_{[\alpha]}\right)=\alpha!\left(\frac{1}{\alpha!} h_{\lambda(\alpha)}\right)=h_{\lambda(\alpha)}=\chi(\nch_\alpha).$$

It transpires that under $\jota$, noncommutative ribbon Schur functions and immaculate functions indexed by integer partitions in $\NSym$ map to source skew Schur functions in $\NCSym$. However, before we prove this result we will narrow our focus.

Recall that in $\Sym$ when $\lambda/ \mu$ is a \emph{ribbon diagram}, namely a skew diagram that is edgewise connected with no $2\times2$ subdiagram,  we often denote $s_{\lambda / \mu}$ by $r_\alpha$ where $\alpha$ is the composition given by the row lengths of $\lambda / \mu$ taken in order from top to bottom, and we call $r_\alpha$ a \emph{ribbon Schur function}. Similarly let us denote $s_{[\lambda / \mu]}$, when $\lambda / \mu$ is a ribbon diagram \svw{corresponding to} $\alpha$, by $r_{[\alpha]}$ and call it a \emph{source ribbon Schur function in noncommuting variables}. 

\begin{corollary}\label{cor:ribbons} We have the following for compositions $\alpha$ and $\beta$.
\begin{enumerate}[1)]
\item $r_{[\alpha]} = (-1)^{\ell(\alpha)} \sum _{\beta \succcurlyeq \alpha} \frac{(-1)^{\ell (\beta)}}{\beta !} h_{[\beta ]}$
\item $r_{[\alpha]} r_{[\beta]} = r_{[\alpha \cdot \beta]} + r_{[\alpha \odot \beta]}$
\item $\rho (r_{[\alpha]}) = r_\alpha$
\end{enumerate}
\end{corollary}

\begin{proof} The second and third parts follow immediately by Theorem~\ref{the:prod} and Proposition~\ref{prop:rhoskew}. For the first part, note that if $\alpha=\alpha_1 \alpha _2 \cdots  \alpha _{\ell(\alpha)}$, then by the second part $$r_{[\alpha]} = r_{[\alpha _1]} r _{[\alpha _ 2 \cdots \alpha _{\ell(\alpha)}]} -  r_{[(\alpha_1 +\alpha _2) \alpha _3 \cdots \alpha _{\ell(\alpha)}]}$$and together with induction on $\ell(\alpha)$ gives the desired result.
\end{proof}

\begin{proposition}\label{prop:iotaacts} We have the following for composition $\alpha$ and integer partition $\lambda$.
\begin{enumerate}[1)]
\item $\jota (\ncr _\alpha) = r_{[\alpha]}$
\item $\jota (\Sn _\lambda) = s_{[\lambda]}$
\end{enumerate}
\end{proposition}

\begin{proof}
Note that
$$\jota(\ncr_{\alpha}) = (-1)^{\ell(\alpha)} \sum _{\beta \succcurlyeq \alpha} (-1)^{\ell (\beta)}  \jota(\nch_{\beta })  = (-1)^{\ell(\alpha)} \sum _{\beta \succcurlyeq \alpha} \frac{(-1)^{\ell (\beta)}}{\beta !} h_{[\beta ]} = r_{[\alpha]}$$by Corollary~\ref{cor:ribbons} and
$$\jota(\Sn _{\lambda})=\sum_{\varepsilon\in \Sn_{\ell(\lambda)}} {\rm sgn}(\varepsilon) \jota(\nch_{\lambda+\varepsilon-{\id}}) = \sum_{\varepsilon\in \Sn_{\ell(\lambda)}}  \frac{{\rm sgn}(\varepsilon)}{(\lambda+\varepsilon-{\id})!}h_{[\lambda+\varepsilon-{\id}]} = s_{[\lambda]}$$by Equation~\eqref{eq:newskew}.
\end{proof}

To conclude, note that in these cases, our functions in $\NCSym$ hence immediately inherit results from $\NSym$ such as the right Pieri rule for immaculate functions \cite{BBSSZ} or the product rule for noncommutative ribbon Schur functions
$$\ncr_{\alpha} \ncr_{\beta} = \ncr_{\alpha \cdot \beta} + \ncr_{\alpha \odot \beta}$$that recovers 
$$r_{[\alpha]} r_{[\beta]} = r_{[\alpha \cdot \beta]} + r_{[\alpha \odot \beta]}.$$

\section{Acknowledgments}\label{sec:ack}  \svwrev{The authors would like to thank  Darij Grinberg, Bruce Sagan, Yannic Vargas and Victor Wang for helpful suggestions, and the referees for excellent recommendations, thoughtful comments, and careful attention to our paper. In particular, they would like to thank one referee for suggesting further avenues and directing us to the thesis of Anouk Bergeron-Brleck, and thank the other referee for letting us use their improved prose in the proof of Theorem~\ref{the:prod}, and for their improvements and simplifications to the proofs of Theorem~\ref{the:ncschur}, Corollary~\ref{cor:permbases}, Theorem~\ref{the:deltaact} and Lemma~\ref{lem:swapmon}.}

\bibliographystyle{plain}

\def\cprime{$'$}

\end{document}